\documentclass[a4paper,10pt]{amsart}

\usepackage[utf8]{inputenc}
\usepackage{url}

\usepackage[expansion=false]{microtype}
\usepackage{amssymb,amsmath,amsopn,amsxtra,amsthm,amsfonts}
\usepackage{xr}
\usepackage[mathcal]{euscript}
\usepackage{mathalfa}
\usepackage{txfonts}
\usepackage{todonotes}
\usepackage{xy}
\input xy
\xyoption{all}
\usepackage{tikz}
\usetikzlibrary{matrix,arrows,cd}

\usepackage[pdfusetitle,unicode]{hyperref}

\numberwithin{equation}{section}
\makeatletter
\let\c@equation\c@subsection

\makeatother

\newtheorem{cor}[subsection]{Corollary}
\newtheorem{lem}[subsection]{Lemma}
\newtheorem{prop}[subsection]{Proposition}

\newtheorem{thm}[subsection]{Theorem}

\theoremstyle{definition}
\newtheorem{defn}[subsection]{Definition}
\newtheorem{rem}[subsection]{Remark}

\newtheorem{ex}[subsection]{Example}
\newtheorem{question}[subsection]{Question}

\theoremstyle{remark}

\newcommand{\secref}[1]{Section~\ref{#1}}

\renewcommand{\eqref}[1]{(\ref{#1})}

\usepackage{tikz}
\usetikzlibrary{matrix}
\usepackage{tikz-cd}
\tikzset{shorten <>/.style={shorten >=#1,shorten <=#1}}

\usepackage{fixltx2e}

\usepackage{xspace}

\usepackage{xifthen}

\usepackage{xparse}

\newcommand{\nc}{\newcommand}
\nc{\renc}{\renewcommand}
\nc{\ssec}{\subsection}
\nc{\sssec}{\subsubsection}
\nc{\on}{\operatorname}
\nc{\term}[1]{#1\xspace}

\DeclareMathSymbol{A}{\mathalpha}{operators}{`A}
\DeclareMathSymbol{B}{\mathalpha}{operators}{`B}
\DeclareMathSymbol{C}{\mathalpha}{operators}{`C}
\DeclareMathSymbol{D}{\mathalpha}{operators}{`D}
\DeclareMathSymbol{E}{\mathalpha}{operators}{`E}
\DeclareMathSymbol{F}{\mathalpha}{operators}{`F}
\DeclareMathSymbol{G}{\mathalpha}{operators}{`G}
\DeclareMathSymbol{H}{\mathalpha}{operators}{`H}
\DeclareMathSymbol{I}{\mathalpha}{operators}{`I}
\DeclareMathSymbol{J}{\mathalpha}{operators}{`J}
\DeclareMathSymbol{K}{\mathalpha}{operators}{`K}
\DeclareMathSymbol{L}{\mathalpha}{operators}{`L}
\DeclareMathSymbol{M}{\mathalpha}{operators}{`M}
\DeclareMathSymbol{N}{\mathalpha}{operators}{`N}
\DeclareMathSymbol{O}{\mathalpha}{operators}{`O}
\DeclareMathSymbol{P}{\mathalpha}{operators}{`P}
\DeclareMathSymbol{Q}{\mathalpha}{operators}{`Q}
\DeclareMathSymbol{R}{\mathalpha}{operators}{`R}
\DeclareMathSymbol{S}{\mathalpha}{operators}{`S}
\DeclareMathSymbol{T}{\mathalpha}{operators}{`T}
\DeclareMathSymbol{U}{\mathalpha}{operators}{`U}
\DeclareMathSymbol{V}{\mathalpha}{operators}{`V}
\DeclareMathSymbol{W}{\mathalpha}{operators}{`W}
\DeclareMathSymbol{X}{\mathalpha}{operators}{`X}
\DeclareMathSymbol{Y}{\mathalpha}{operators}{`Y}
\DeclareMathSymbol{Z}{\mathalpha}{operators}{`Z}

\nc{\sA}{\ensuremath{\mathcal{A}}\xspace}
\nc{\sB}{\ensuremath{\mathcal{B}}\xspace}
\nc{\sC}{\ensuremath{\mathcal{C}}\xspace}
\nc{\sD}{\ensuremath{\mathcal{D}}\xspace}
\nc{\sE}{\ensuremath{\mathcal{E}}\xspace}
\nc{\sF}{\ensuremath{\mathcal{F}}\xspace}
\nc{\sG}{\ensuremath{\mathcal{G}}\xspace}
\nc{\sH}{\ensuremath{\mathcal{H}}\xspace}
\nc{\sI}{\ensuremath{\mathcal{I}}\xspace}
\nc{\sJ}{\ensuremath{\mathcal{J}}\xspace}
\nc{\sK}{\ensuremath{\mathcal{K}}\xspace}
\nc{\sL}{\ensuremath{\mathcal{L}}\xspace}
\nc{\sM}{\ensuremath{\mathcal{M}}\xspace}
\nc{\sN}{\ensuremath{\mathcal{N}}\xspace}
\nc{\sO}{\ensuremath{\mathcal{O}}\xspace}
\nc{\sP}{\ensuremath{\mathcal{P}}\xspace}
\nc{\sQ}{\ensuremath{\mathcal{Q}}\xspace}
\nc{\sR}{\ensuremath{\mathcal{R}}\xspace}
\nc{\sS}{\ensuremath{\mathcal{S}}\xspace}
\nc{\sT}{\ensuremath{\mathcal{T}}\xspace}
\nc{\sU}{\ensuremath{\mathcal{U}}\xspace}
\nc{\sV}{\ensuremath{\mathcal{V}}\xspace}
\nc{\sW}{\ensuremath{\mathcal{W}}\xspace}
\nc{\sX}{\ensuremath{\mathcal{X}}\xspace}
\nc{\sY}{\ensuremath{\mathcal{Y}}\xspace}
\nc{\sZ}{\ensuremath{\mathcal{Z}}\xspace}

\nc{\bA}{\ensuremath{\mathbf{A}}\xspace}
\nc{\bB}{\ensuremath{\mathbf{B}}\xspace}
\nc{\bC}{\ensuremath{\mathbf{C}}\xspace}
\nc{\bD}{\ensuremath{\mathbf{D}}\xspace}
\nc{\bE}{\ensuremath{\mathbf{E}}\xspace}
\nc{\bF}{\ensuremath{\mathbf{F}}\xspace}
\nc{\bG}{\ensuremath{\mathbf{G}}\xspace}
\nc{\bH}{\ensuremath{\mathbf{H}}\xspace}
\nc{\bI}{\ensuremath{\mathbf{I}}\xspace}
\nc{\bJ}{\ensuremath{\mathbf{J}}\xspace}
\nc{\bK}{\ensuremath{\mathbf{K}}\xspace}
\nc{\bL}{\ensuremath{\mathbf{L}}\xspace}
\nc{\bM}{\ensuremath{\mathbf{M}}\xspace}
\nc{\bN}{\ensuremath{\mathbf{N}}\xspace}
\nc{\bO}{\ensuremath{\mathbf{O}}\xspace}
\nc{\bP}{\ensuremath{\mathbf{P}}\xspace}
\nc{\bQ}{\ensuremath{\mathbf{Q}}\xspace}
\nc{\bR}{\ensuremath{\mathbf{R}}\xspace}
\nc{\bS}{\ensuremath{\mathbf{S}}\xspace}
\nc{\bT}{\ensuremath{\mathbf{T}}\xspace}
\nc{\bU}{\ensuremath{\mathbf{U}}\xspace}
\nc{\bV}{\ensuremath{\mathbf{V}}\xspace}
\nc{\bW}{\ensuremath{\mathbf{W}}\xspace}
\nc{\bX}{\ensuremath{\mathbf{X}}\xspace}
\nc{\bY}{\ensuremath{\mathbf{Y}}\xspace}
\nc{\bZ}{\ensuremath{\mathbf{Z}}\xspace}

\nc{\dA}{\ensuremath{\mathds{A}}\xspace}
\nc{\dB}{\ensuremath{\mathds{B}}\xspace}
\nc{\dC}{\ensuremath{\mathds{C}}\xspace}
\nc{\dD}{\ensuremath{\mathds{D}}\xspace}
\nc{\dE}{\ensuremath{\mathds{E}}\xspace}
\nc{\dF}{\ensuremath{\mathds{F}}\xspace}
\nc{\dG}{\ensuremath{\mathds{G}}\xspace}
\nc{\dH}{\ensuremath{\mathds{H}}\xspace}
\nc{\dI}{\ensuremath{\mathds{I}}\xspace}
\nc{\dJ}{\ensuremath{\mathds{J}}\xspace}
\nc{\dK}{\ensuremath{\mathds{K}}\xspace}
\nc{\dL}{\ensuremath{\mathds{L}}\xspace}
\nc{\dM}{\ensuremath{\mathds{M}}\xspace}
\nc{\dN}{\ensuremath{\mathds{N}}\xspace}
\nc{\dO}{\ensuremath{\mathds{O}}\xspace}
\nc{\dP}{\ensuremath{\mathds{P}}\xspace}
\nc{\dQ}{\ensuremath{\mathds{Q}}\xspace}
\nc{\dR}{\ensuremath{\mathds{R}}\xspace}
\nc{\dS}{\ensuremath{\mathds{S}}\xspace}
\nc{\dT}{\ensuremath{\mathds{T}}\xspace}
\nc{\dU}{\ensuremath{\mathds{U}}\xspace}
\nc{\dV}{\ensuremath{\mathds{V}}\xspace}
\nc{\dW}{\ensuremath{\mathds{W}}\xspace}
\nc{\dX}{\ensuremath{\mathds{X}}\xspace}
\nc{\dY}{\ensuremath{\mathds{Y}}\xspace}
\nc{\dZ}{\ensuremath{\mathds{Z}}\xspace}

\nc{\bbA}{\ensuremath{\mathbb{A}}\xspace}
\nc{\bbB}{\ensuremath{\mathbb{B}}\xspace}
\nc{\bbC}{\ensuremath{\mathbb{C}}\xspace}
\nc{\bbD}{\ensuremath{\mathbb{D}}\xspace}
\nc{\bbE}{\ensuremath{\mathbb{E}}\xspace}
\nc{\bbF}{\ensuremath{\mathbb{F}}\xspace}
\nc{\bbG}{\ensuremath{\mathbb{G}}\xspace}
\nc{\bbH}{\ensuremath{\mathbb{H}}\xspace}
\nc{\bbI}{\ensuremath{\mathbb{I}}\xspace}
\nc{\bbJ}{\ensuremath{\mathbb{J}}\xspace}
\nc{\bbK}{\ensuremath{\mathbb{K}}\xspace}
\nc{\bbL}{\ensuremath{\mathbb{L}}\xspace}
\nc{\bbM}{\ensuremath{\mathbb{M}}\xspace}
\nc{\bbN}{\ensuremath{\mathbb{N}}\xspace}
\nc{\bbO}{\ensuremath{\mathbb{O}}\xspace}
\nc{\bbP}{\ensuremath{\mathbb{P}}\xspace}
\nc{\bbQ}{\ensuremath{\mathbb{Q}}\xspace}
\nc{\bbR}{\ensuremath{\mathbb{R}}\xspace}
\nc{\bbS}{\ensuremath{\mathbb{S}}\xspace}
\nc{\bbT}{\ensuremath{\mathbb{T}}\xspace}
\nc{\bbU}{\ensuremath{\mathbb{U}}\xspace}
\nc{\bbV}{\ensuremath{\mathbb{V}}\xspace}
\nc{\bbW}{\ensuremath{\mathbb{W}}\xspace}
\nc{\bbX}{\ensuremath{\mathbb{X}}\xspace}
\nc{\bbY}{\ensuremath{\mathbb{Y}}\xspace}
\nc{\bbZ}{\ensuremath{\mathbb{Z}}\xspace}

\nc{\mrm}[1]{\ensuremath{\mathrm{#1}}\xspace}
\nc{\mbf}[1]{\ensuremath{\mathbf{#1}}\xspace}
\nc{\mcal}[1]{\ensuremath{\mathcal{#1}}\xspace}
\nc{\msc}[1]{\ensuremath{\mathscr{#1}}\xspace}

\renc{\bar}[1]{\overline{#1}}

\let\sectsign\S
\let\S\relax

\nc{\sub}{\subset}
\nc{\too}{\longrightarrow}
\nc{\hook}{\hookrightarrow}
\nc*{\hooklongrightarrow}{\ensuremath{\lhook\joinrel\relbar\joinrel\rightarrow}}
\nc{\hooklong}{\hooklongrightarrow}
\nc{\twoheadlongrightarrow}{\relbar\joinrel\twoheadrightarrow}
\nc{\shiso}{\approx}
\nc{\isoto}{\xrightarrow{\sim}}
\nc{\isofrom}{\xleftarrow{\sim}}
\renc{\ge}{\geqslant}
\renc{\le}{\leqslant}
\renc{\geq}{\geqslant}
\renc{\leq}{\leqslant}

\nc{\id}{\mathrm{id}}
\nc{\can}{\mathrm{can}}

\let\Im\relax
\DeclareMathOperator{\Im}{\mathrm{Im}}
\DeclareMathOperator{\rk}{\mathrm{rk}}
\DeclareMathOperator{\Hom}{\mathrm{Hom}}
\nc{\uHom}{\underline{\smash{\Hom}}}

\DeclareMathOperator{\Aut}{\mathrm{Aut}}
\DeclareMathOperator{\End}{\mathrm{End}}

\nc{\Pre}{\mathrm{PSh}{}}
\nc{\uEnd}{\underline{\smash{\End}}}

\renc{\lim}{\operatorname*{lim}}
\nc{\colim}{\operatorname*{colim}}
\nc{\Cofib}{\on{Cofib}}
\nc{\Fib}{\on{Fib}}
\nc{\initial}{\varnothing}
\nc{\op}{\mathrm{op}}

\renc{\coprod}{\sqcup}

\nc{\bDelta}{\mbf{\Delta}}
\nc{\DM}{\mbf{DM}}
\nc{\eff}{\mathrm{eff}}
\nc{\bir}{\mathrm{bir}}
\nc{\SB}{\mathrm{SB}}
\nc{\veff}{\mathrm{veff}}
\nc{\cyc}{{\mrm{cyc}}}
\nc{\Cor}{{\mrm{Cor}}}
\nc{\corr}{{\on{corr}}}
\nc{\fet}{{\mrm{f\acute et}}}
\nc{\fsyn}{{\mrm{fsyn}}}
\nc{\fflat}{{\mrm{fflat}}}
\nc{\syn}{{\mrm{syn}}}
\nc{\Br}{{\mrm{Br}}}
\nc{\Perf}{\mathcal{P}\mrm{erf}}
\nc{\QCoh}{\mathcal{Q}\mrm{Coh}}
\nc{\Az}{\mathcal{A}\mrm{z}}
\nc{\Pic}{\mrm{Pic}}
\nc{\Nrd}{\mrm{Nrd}}
\nc{\perf}{\mrm{perf}}
\nc{\oblv}{\on{oblv}}
\nc{\exact}{\on{exact}}

\nc{\F}{{\on{F}}}
\nc{\clopen}{{\mrm{clopen}}}
\nc{\B}{\mrm{B}}
\nc{\D}{\mrm{D}}
\nc{\Fin}{\on{Fin}}
\nc{\Cut}{\on{Cut}}
\nc{\Cart}{\on{Cart}}
\nc{\pairs}{\mathsf{pairs}}
\nc{\Pairs}{\mathrm{Pair}}
\nc{\Trip}{\mathrm{Trip}}
\nc{\Lab}{\mathrm{Lab}}
\nc{\coCart}{\mathrm{coCart}}
\nc{\RKE}{\mathrm{RKE}}
\nc{\strict}{\mathrm{strict}}
\nc{\Emb}{\mathrm{Emb}}
\nc{\EMB}{\mathcal{E}\mathrm{mb}}
\nc{\Split}{\mathrm{Split}}
\nc{\Set}{\mathrm{Set}}
\nc{\sSets}{\mathrm{sSets}}
\nc{\pb}{\mathrm{pb}}
\nc{\lci}{\mathrm{lci}}
\nc{\fib}{\mathrm{fib}}
\nc{\cofib}{\mathrm{cofib}}
\nc{\diff}{\mrm{diff}}
\nc{\gp}{\mrm{gp}}
\nc{\ind}{\mrm{ind}}
\nc{\chr}{\mrm{char}}
\nc{\mgp}{\mrm{mot-gp}}
\nc{\FSyn}{\mrm{FSyn}}
\nc{\FFlat}{{\mrm{FFlat}}}
\nc{\FEt}{\mrm{FEt}}
\nc{\Spc}{\mrm{Spc}}
\nc{\Ob}{\mrm{Ob}}
\nc{\Spt}{\mrm{Spt}}
\nc{\tors}{\mrm{tors}}
\nc{\T}{\bT}
\nc{\suspinf}{\Sigma^\infty}
\nc{\h}{\mrm{h}}
\nc{\uhom}{\underline{\mathrm{Hom}}}
\nc{\umap}{\underline{\mathrm{Maps}}}
\nc{\Map}{\mathrm{Map}}
\nc{\map}{\operatorname{map}}           
\nc{\Autom}{\mathrm{Aut}}
\renc{\H}{\bH}
\nc{\Einfty}{{\sE_\infty}}
\nc{\Eone}{{\sE_1}}
\nc{\Stab}{\mrm{Stab}}
\nc{\lax}{{\mrm{lax}}}
\nc{\cocart}{{\mrm{cocart}}}
\nc{\Sch}{\mrm{Sch}}
\nc{\dSch}{\mrm{dSch}}
\nc{\Aff}{\mrm{Aff}}
\nc{\SmAff}{\mrm{SmAff}}
\nc{\dAff}{\mrm{dAff}}
\nc{\Fr}{\on{Fr}}
\nc{\A}{\mathbf{A}}
\nc{\Pp}{\mathbf{P}}
\nc{\Ss}{\mathbf{S}}
\nc{\N}{\mathbf{N}}
\nc{\Z}{\mathbf{Z}}
\nc{\Q}{\mathbf{Q}}
\nc{\Oo}{\mathcal{O}} 
\nc{\Aa}{\mathcal{A}} 
\nc{\Bb}{\mathcal{B}} 
\nc{\Ee}{\mathcal{E}}
\nc{\Ff}{\mathcal{F}} 
\nc{\Gg}{\mathcal{G}} 
\nc{\red}{{\on{red}}}
\nc{\Voev}{{\on{Voev}}}
\nc{\Corr}{\mrm{Corr}}
\nc{\Span}{\mathbf{Corr}}
\nc{\Gap}{\mrm{Gap}}
\nc{\Filt}{\mrm{Filt}}
\nc{\Corrfr}{\Corr^{\fr}}
\nc{\Corrvfr}{\Corr^{\Vfr}}
\nc{\Spec}{\on{Spec}}
\nc{\Sm}{\mrm{Sm}}
\nc{\QSm}{\mrm{QSm}}
\nc{\Gm}{\mathbf{G}_{\mrm{m}}}
\renc{\P}{\bP}
\nc{\nis}{\mathrm{nis}}
\nc{\KH}{\mathrm{KH}}
\nc{\bil}{\mathrm{bil}}
\nc{\Zar}{\mathrm{Zar}}
\nc{\zar}{\mathrm{zar}}
\nc{\Nis}{\mathrm{Nis}}
\nc{\et}{\mathrm{\acute et}}
\nc{\all}{\mathrm{all}}
\nc{\fold}{\mathrm{fold}}
\nc{\Fun}{\mathrm{Fun}}
\nc{\Ho}{\mathrm{Ho}}
\nc{\Segal}{\mathrm{Segal}}
\nc{\Mon}{\mrm{Mon}{}}
\nc{\Ab}{\mrm{Ab}}
\nc{\Gr}{\mrm{Gr}}
\nc{\Sh}{\on{Sh}}
\nc{\M}{\mrm{M}}
\nc{\Lhtp}{L_{\A^1}}
\nc{\Lzar}{L_{\Zar}}
\nc{\Lnis}{L_{\Nis}}
\nc{\Lmot}{L_{\mrm{mot}}}
\nc{\mot}{\mrm{mot}}
\nc{\SH}{\mbf{SH}}
\nc{\HH}{\mbf{H}}
\nc{\RR}{\mbf{R}}
\nc{\CC}{\mbf{C}}
\nc{\Mod}{\mrm{Mod}}

\nc{\MonUnit}{\mbf{1}}
\nc{\tr}{\on{tr}}
\nc{\vop}{\mrm{vop}}
\nc{\fr}{{\on{fr}}}
\nc{\Ar}{\mrm{Ar}}
\nc{\Vfr}{\on{Vfr}}
\nc{\frdiff}{{\on{frdiff}}}
\nc{\frGys}{\on{frGys}}
\nc{\SHfr}{\SH^{\fr}}
\nc{\SHfrdiff}{\SH^{\frdiff}}
\nc{\SHfrGys}{\SH^{\frGys}}
\nc{\InftyCat}{\infty\textnormal{-}\mrm{Cat}}
\nc{\TriCat}{\mathrm{TriCat}}
\nc{\Cat}{\mathrm{1\textnormal{-}Cat}}
\nc{\Th}{\on{Th}}
\def\G{\bG}
\nc{\CMon}{\mrm{CMon}{}}
\nc{\CAlg}{\mrm{CAlg}{}}
\nc{\MGL}{\mrm{MGL}}
\nc{\PMGL}{\mathrm{PMGL}}
\nc{\KGL}{\mrm{KGL}}
\nc{\kgl}{\mrm{kgl}}
\nc{\MSL}{\mrm{MSL}}
\nc{\MSp}{\mrm{MSp}}
\nc{\Seg}{\mrm{Seg}{}}
\nc{\Tw}{\mrm{Tw}}
\nc{\sslash}{/\mkern-6mu/}
\nc{\PrL}{\mrm{Pr}^\mrm{L}}
\nc{\PrR}{\mrm{Pr}^\mrm{R}}
\nc{\pr}{\mrm{pr}}
\nc{\efr}{\mrm{efr}}
\nc{\nfr}{\mrm{nfr}}
\nc{\dfr}{\mrm{fr}}
\nc{\tfr}{\mrm{tfr}}
\nc{\Vect}{\mathcal{V}\mrm{ect}}
\nc{\sVect}{\mrm{sVect}}
\nc{\Alg}{\mrm{Alg}}
\nc{\fix}{\mrm{fix}}
\nc{\Hilb}{\mathrm{Hilb}}
\nc{\flci}{\mathrm{flci}}
\nc{\Isom}{\mathrm{Isom}}
\nc{\GL}{\mathrm{GL}}
\nc{\BGL}{\mathrm{BGL}}
\nc{\PGL}{\mathrm{PGL}}
\nc{\SL}{\mathrm{SL}}
\nc{\Sp}{\mathrm{Sp}}
\nc{\fin}{\mathrm{fin}}
\nc{\cl}{\mathrm{cl}}
\nc{\cn}{\mathrm{cn}}
\nc{\sm}{\mathrm{sm}}
\nc{\heart}{\heartsuit}
\renc{\o}{\mrm{or}}
\nc{\GW}{\mrm{GW}}
\nc{\ev}{\mrm{ev}}
\nc{\FSYN}{\mathcal{FS}\mrm{yn}}
\nc{\FFLAT}{\mathcal{FF}\mrm{lat}}
\nc{\mrk}{\mrm{mrk}}%
\nc{\FFmrk}{\mathcal{FF}\mrm{lat}^{\mrk}}
\nc{\FFnu}{\mathcal{FF}\mrm{lat}^{\mrm{nu}}}
\nc{\FFbas}{\mathcal{FF}\mrm{lat}^{\mrm{bas}}}
\nc{\FQSM}{\mathcal{FQS}\mathrm{m}}
\nc{\Quot}{\mathrm{Quot}}    
\nc{\COH}{\mathcal{C}\mathrm{oh}}
\let\phi\varphi

\nc{\robber}{\mathcal{R}}%
\nc{\mv}{\mrm{mv}}
\nc{\const}{\mrm{const}}
\nc{\robbermv}{\robber^{\mv}}
\nc{\robberconst}{\robber^{\const}}
\nc{\robbernot}{\robber_0}%
\nc{\sectionmv}{i^{\mv}}%
\nc{\sectionconst}{i^{\const}}%
\nc{\sectionnot}{i}%
\nc{\st}{\mathrm{st}}

\nc{\inftyCat}{\term{$\infty$-category}}
\nc{\inftyCats}{\term{$\infty$-categories}}

\nc{\inftyOneCat}{\term{$(\infty,1)$-category}}
\nc{\inftyOneCats}{\term{$(\infty,1)$-categories}}

\nc{\inftyGrpd}{\term{$\infty$-groupoid}}
\nc{\inftyGrpds}{\term{$\infty$-groupoids}}

\nc{\inftyTop}{\term{$\infty$-topos}}
\nc{\inftyTops}{\term{$\infty$-toposes}}

\nc{\inftyTwoCat}{\term{$(\infty,2)$-category}}
\nc{\inftyTwoCats}{\term{$(\infty,2)$-categories}}

\title{Twisted K-theory in motivic homotopy theory}

\author[E. Elmanto]{Elden Elmanto}
\address{IMJ-PRG and CNRS\\
SU  4 place Jussieu\\
Case 247\\
75252 Paris\\
France}
\email{\href{mailto:eldenelmanto@gmail.com}{eldenelmanto@gmail.com}}
\urladdr{\url{https://www.eldenelmanto.com/}}

\author[D. Nardin]{Denis Nardin}
\address{Fakultät für Mathematik\\
Universität Regensburg\\
93040 Regensburg\\
Deutschland}
\email{\href{mailto:denis.nardin@ur.de}{denis.nardin@ur.de}}
\urladdr{\url{https://homepages.uni-regensburg.de/~nad22969/}}

\author[M. Yakerson]{Maria Yakerson}
\address{Institute for Mathematical Research (FIM)\\
ETH Z\"urich \\
R\"amistr. 101\\  
8092 Z\"urich\\
Switzerland}
\email{\href{mailto:maria.yakerson@math.ethz.ch}{maria.yakerson@math.ethz.ch}}
\urladdr{\url{https://www.muramatik.com}}

\date{\today}

\begin{document}

\maketitle

\begin{abstract}
In this paper, we study twisted algebraic $K$-theory from a motivic viewpoint. For a smooth variety $X$ over a field of characteristic zero and an Azumaya algebra $\Aa$ over $X$, we construct the  $\Aa$-twisted motivic spectral sequence, by computing the slices of the motivic twisted algebraic $K$-theory spectrum as a twisted form of motivic cohomology. This generalizes previous results due to Kahn-Levine where $\Aa$ is assumed to be pulled back from a base field. Our methods use interaction between the slice filtration and birational geometry. 
Along the way, we prove a representability result, expressing the motivic space of twisted $K$-theory as an extension of the twisted Grassmannian by the sheaf of ``twisted integers". This leads to a proof of cdh descent and Milnor excision for twisted homotopy $K$-theory.
\end{abstract}

\parskip 0.2cm
\parskip 0pt
\tableofcontents
\parskip 0.2cm
\vspace{-2em}

 \section{Introduction}
 
Twists naturally occur in mathematics whenever locally-defined entities fail to glue to a global object. We are particularly interested in twists by Brauer classes in the context of algebraic geometry. We first recall some of the cases when such twists appear, see also~\cite{lieblich-ubiquity} for more examples.

\begin{ex} Suppose that $p\colon X \rightarrow Y$ is a projective morphism such that all of its geometric fibers are isomorphic to a projective space of fixed rank. Examples of such morphisms arise from taking the projectivization of a vector bundle $\sE$ on $Y$. However, there are many more examples: $X$ is a \emph{Severi-Brauer} scheme if 
it is \'etale-locally on $Y$ isomorphic to $\P^n \times Y$. The failure of $p$ to be a projective bundle is measured by the associated Brauer class $\alpha \in \mathrm{Br}(Y)$, which  is an \'etale $2$-cocycle valued in $\Gm$. 
\end{ex}

\begin{ex}\label{ex:2} If $X$ is a smooth projective variety over a field, one can consider a moduli space of sheaves (with stability or semistability conditions) $\mathcal{N}$. Such a moduli space usually exists a coarse space but may not exist as a fine moduli space. 
 The problem is that the universal sheaf exists locally but may not glue together to form a global sheaf. In certain cases, C\u{a}ld\u{a}raru proves that a Brauer class $\alpha$ on $\mathcal{N}$ obstructs the existence of a universal sheaf and that a universal $\alpha$-\emph{twisted} sheaf exists instead. He further proves a derived equivalence between the perfect complexes on $X$ and twisted perfect complexes on $\mathcal{N}$ \cite[Theorem 5.5.1]{caldararu-thesis}.
\end{ex}

The goal of this paper is to improve certain tools for studying twisted version of algebraic K-theory. For a smooth $k$-scheme $X$ and $\Aa$ an Azumaya algebra over $X$, the \emph{$\Aa$-twisted K-theory} space $K^\Aa(X)$ is the K-theory space of the stable $\infty$-category of $\Aa$-twisted perfect complexes on $X$. By construction, the $\Aa$-twisted K-theory only depends on the equivalence class of $\Aa$ in the Brauer group of $X$; in other words, it is \emph{Morita invariant}. The $\Aa$-twisted K-theory groups $K^\Aa_*(X)$, defined as the homotopy groups of the space $K^\Aa(X)$, first appeared in Quillen's computation of the $K$-theory of Severi-Brauer varieties~\cite{Quillen:1973}.

\begin{rem}
The space $K^\Aa(X)$ can be defined more generally for any Brauer class $\Aa$, not necessarily representable by an Azumaya algebra. While in this paper we will concentrate for simplicity on the case when such a representing object exists, all the proofs will work in the more general case as well, as they always proceed by reduction to the case where $\Aa$ is pulled back from a field where all Brauer classes are representable.
\end{rem}

 When $\Aa$ is trivial in the Brauer group, one of the main tools for  accessing the $K$-theory groups $K_*(X)$ is the \emph{motivic spectral sequence}
 \begin{equation}\label{eq:mss}
E_2^{p,q}=CH^{-q}(X;-p-q)\Rightarrow K_{-p-q}(X),
 \end{equation}
where $CH^*(X;*)$ are Bloch's \emph{higher Chow groups}. The spectral sequence~\eqref{eq:mss} has many implications: for example, it is a key ingredient in Voevodsky's proof of the Quillen-Lichtenbaum conjectures \cite{Voevodsky:2008}. The spectral sequence degenerates after rationalization and implies (a variant of) the Grothendieck-Riemann-Roch theorem:
\[
K_0(X)_{\Q} \simeq \bigoplus^{\dim(X)}_{i=0} CH^i(X)_{\Q}.
\]

Bloch's higher Chow groups provide a cohomology theory for schemes which satisfies Beilinson's desiderata for a theory of ``motivic cohomology"; see for example \cite{FSV-intro} for the list of properties. Voevodsky's construction of motivic cohomology groups compares with Bloch's higher Chow groups via an isomorphism
\[
CH^{q}(X;2q-p) \simeq H^p(X;\Z(q))
\]
for smooth schemes over a perfect field \cite{Voevodsky:2002b}. In light of the last isomorphism, the main result of this paper is an extension of~\eqref{eq:mss} to the twisted setting.

\begin{thm}[Corollary~\ref{cor:main}]\label{thm:main-intro} Let $k$ be a field of characteristic zero, $X$ a smooth $k$-scheme and $\Aa$ an Azumaya algebra over $X$. Then there is a strongly convergent spectral sequence
    \begin{equation}\label{eq:tw-mss}
    E^{p,q}_2=H^{p-q}_{\mot}(X;\Z^\Aa(-q))\Rightarrow K_{-p-q}^\Aa(X),
    \end{equation}
    where $E_2$-terms are the $\Aa$-twisted motivic cohomology groups.
    For a field $k$ of characteristic $p$, the same result holds after inverting $p$.
    \end{thm}

When $\Aa$ is pulled back from the base field $k$, the spectral sequence~\eqref{eq:tw-mss} is one of the main results of the work of  Kahn and Levine on twisted K-theory~\cite[Theorem~1]{levine-kahn}. In this case, the $\Aa$-twisted motivic cohomology groups have a cycle-theoretic interpretation. Namely, they can be defined as $\Aa$-twisted higher Chow groups: the generators are algebraic cycles labelled by ``$\Aa$-twisted integers" (see Proposition~\ref{prop:cycle-model-of-Z^A}). The sheaf of $\Aa$-twisted integers $\Z^\Aa$ is the sheafification of the presheaf $K_0^\Aa$, and the twist encodes the index of $\Aa$. In the case of a general $\Aa$, the cycle-theoretic description holds locally, and we expect that it applies globally as well (see Remark~\ref{rem:nis-local}) and that the spectral sequence \eqref{eq:tw-mss} coincides with the spectral sequence of \cite[Corollary 6.1.4]{levine-kahn}.

The spectral sequence~\eqref{eq:tw-mss} results from the computation of slices of the motivic spectrum $\KGL^\Aa$, representing twisted algebraic K-theory. Under the assumptions of Theorem~\ref{thm:main-intro}, we compute the slices $s_q$ with respect to Voevodsky's slice filtration as follows (see Theorem~\ref{thm:s_0KGL-general-base}):\begin{equation}\label{eq:tw-slices}
s_q\KGL^{\sA} \simeq \Sigma^q_{\P^1} H\bZ^{\sA} \qquad \forall q \in \bbZ.
\end{equation}

Here $H\bZ^{\sA}$ is the $\Aa$-twisted motivic cohomology spectrum, which we construct in Definition~\ref{def:twisted-mot}. When the Brauer class of $\sA$ is trivial, the spectrum $H\bZ^{\sA}$ coincides with Spitzweck's motivic cohomology spectrum \cite{SpitzweckHZ}, which, in turn, represents Bloch's higher Chow groups for smooth schemes over a field.

The computation of slices of twisted K-theory in~\cite[Theorem~1]{levine-kahn} goes through the twisted version of Levine's homotopy coniveau tower~\cite{Levine:2008}. We simplify some of the crucial steps in the proof by means of the techniques of Bachmann and the first  author developed in~\cite{tom-elden-slices}, which make use of birational geometry. We elaborate more on the differences between our approach and~\cite{levine-kahn} in the end of the Introduction.

An important ingredient in our proof is representability of twisted K-theory, which has its own implications. One of the main results of the foundational work of Morel and Voevodsky on $\A^1$-homotopy theory~\cite{MV} was the representability of algebraic K-theory motivic space by the ind-smooth scheme $\Z \times \Gr_\infty$, where $\Gr_\infty$ is the infinite Grassmannian. This result allows to deduce that the motivic K-theory spectrum $\KGL$ is stable under pullback, which in turn implies  that homotopy $K$-theory satisfies cdh descent \cite{Cisinski} and Milnor excision \cite{EHIK}.

An immediate analogue of representability for twisted K-theory fails, but we show that there is a fiber sequence of motivic spaces (see~Theorem~\ref{thm:fiber seq K^A}):
\begin{equation}\label{eq:fib}
\Gr^\Aa_\infty \to  K^{\Aa} \rightarrow \Z^\Aa.
\end{equation}
This sequence splits Zariski locally, and it splits globally if $\Aa$ is trivial.

A key geometric ingredient in our computation~\eqref{eq:tw-slices} of slices of the motivic $\Aa$-twisted K-theory spectrum $\KGL^\Aa$ is rationality of the ind-scheme $\Gr_\infty^\Aa$. Moreover, we show that the terms of the fiber sequence~\eqref{eq:fib} are stable under pulback and deduce the following result.

\begin{thm}[Theorem~\ref{thm:pull}]\label{thm:intro2} Let $X$ be a regular scheme of finite Krull dimension, and $\Aa$ an Azumaya algebra over $X$. Then the $\Aa$-twisted homotopy K-theory presheaf of spectra 
on the category of $X$-schemes satisfies Milnor excision and cdh descent.
\end{thm}

 \subsection*{Vista} Let $X$ be a smooth $k$-variety with the structure morphism $p\colon X \rightarrow \Spec k$ and $\alpha \in \Br(X)$ be a Brauer class. We can define the \emph{$\alpha$-twisted motive of $X$} to be 
 \[
 M(X, \alpha) := p_{\#}H\Z^{\alpha}.
\]
This motivic spectrum is canonically an $H\Z$-module and thus defines an object in the category of Voevodsky's motives $\DM(k)$. When $X = \Spec k$, this object has been studied in~\cite{levine-kahn}, but it is new otherwise, and so we can ask for ``twisted" analogs of some conjectures and theorems on smooth varieties with many interesting Brauer classes.

One family of examples of such viarieties is given by $K3$ surfaces; see \cite[Chapter 18]{huybrechts-lectures} for a discussion on the Brauer-theoretic aspects of a $K3$ surface. A fascinating result, due to Huybrechts, is that twisted derived equivalent $K3$ surfaces give rise to isomorphic \emph{rational} Chow motives~\cite{h-1,h-2}. Since twists disappear after rationalization, one can then wonder if there is an integral refinement of Huybrechts' results.

\begin{question}\label{quest" K3} If $(X, \alpha)$ and $(X',\alpha')$ are two twisted $K3$ surfaces with equivalent categories of twisted perfect complexes, are their motives $M(X,\alpha)$ and $M(X',\alpha')$ equivalent? 
\end{question}

Answering Question~\ref{quest" K3} requires a better understanding of the twisted spectral sequence~\eqref{eq:mss}, since $(X,\alpha)$ and $(X',\alpha')$ have isomorphic twisted $K$-theory groups. 
It is subject to further investigation by the authors. 

\subsection*{Comparison with other works}\label{sect:compare}

Let us elaborate on the comparison of Theorem~\ref{thm:main-intro} with~\cite[Theorem 1]{levine-kahn}, when the Azumaya algebra  $\sA$ is pulled back from the base field $k$. Then the statement of \cite[Theorem 1]{levine-kahn} differs in two ways. First, in~\cite{levine-kahn} the slices are computed in the category of $\Ss^1$-spectra $\SH^{\Ss^1}(k)$, i.e., the category of $\A^1$-invariant, Nisnevich sheaves of spectra, rather than the category  of  motivic $\P^1$-spectra.
Second, their result only assumes that $k$ is a perfect field, without inverting the characteristic.

In Theorem~\ref{thm:s_0KGL=HZ} we give a different proof of the original result of Kahn-Levine in the same generality as the original: there is no assumption on the characteristic of $k$, as long as $\sA$ is pulled back from the base field. To get a statement in the form presented in \cite{levine-kahn}, i.e. on the level of $\Ss^1$-spectra, one can  use~\cite[Conjecture 3]{Voevodsky:2002c}, which asserts a compatibility of slice towers under the canonical functor $\omega^{\infty}\colon \SH(k) \rightarrow \SH^{\Ss^1}(k)$. This conjecture was proved by Levine~\cite[Theorem 9.0.3]{LevineChow} and has a different proof in~\cite[Theorem~17]{tom-elden-slices} in the flavour of the present paper. 

One of the main points of the present work is to emphasize that the machinery of motivic homotopy theory provides a rather simple candidate for $\sA$-twisted motivic cohomology, see Definition~\ref{def:twisted-mot}. The cohomology groups appearing in Theorem~\ref{thm:main-intro} (and its proof) do not make reference to Bloch's cycle complexes, which present real technical challenges in the theory. In particular, our proof  avoids the key localization lemma~\cite[Lemma 6.3.1]{levine-kahn}. On the other hand, when the twist $\sA$ is trivial, our definition coincides with all other definitions of motivic cohomology (this latter statement, unfortunately, does not avoid the use of moving lemmas).

Another approach for twisting the motivic K-theory spectrum was introduced by Spitzweck and \O stv\ae r in \cite{spitzweck-ostvaer}. Their approach closely follows the formalism of twisted cohomology theories in topology.  
In particular, they constructed a version of of the twisted motivic spectral sequence in \cite[Theorem 5.17]{spitzweck-ostvaer}. However, their formalism only works for twists by classes in the group $H_{\mot}^{3}(X; \Z(1))$. For a general motivic space $X$, this group might not be trivial (for example, for the motivic sphere $\Sigma \P^1$), 
but it is trivial whenever $X$ is a smooth scheme over a perfect field. It would be interesting to compare and unify their approach to twisted K-theory with ours.

\subsection*{Summary} In~\secref{sec:recall} we recall the formation of twisted algebraic $K$-theory, as well as the necessary prerequisites on Brauer groups and Azumaya algebras. In~\secref{sec:motspc}, we study the motivic $K$-theory space and its representability via the twisted Grassmannian.
This leads us to~\secref{sec:motsptw}, where the motivic twisted K-theory spectrum is introduced, and the cdh descent and Milnor excision for twisted homotopy K-theory are proved. Finally, in~\secref{sec:motcohtw} we compute the slices of twisted K-theory and obtain the twisted motivic spectral sequence.

 \subsection*{Acknowledgments} This project has originally started together with Ben Antieau, and we have benefitted a lot from his help. We also thank sincerely Tom Bachmann, Marc Hoyois, Daniel Huybrechts, Marc Levine, and David Rydh for their useful comments and suggestions. Yakerson was supported by the Hermann-Weyl-Instructorship and is grateful to the Institute of Mathematical Research (FIM) and to ETH Z\"urich for providing perfect working conditions.  This project has received funding from the European Research Council (ERC) under the European Union's Horizon 2020 research and innovation programme (grant agreement No. 101001474).

 \subsection*{Notation}
  $S$ denotes the quasi-compact quasi-separated base scheme,  $\Sch_S$ denotes the category of $S$-schemes, $\Sm_S$ that of smooth $S$-schemes. We denote by $\Pre(\Sm_S)$ the $\infty$-category of presheaves of spaces on $\Sm_S$, $\H(S)$ the $\infty$-category of motivic spaces over $S$, and $\SH(S)$ that of motivic spectra, see \cite[\sectsign 2.2 and \sectsign 4.1]{norms}.
The Tate sphere $\T \in \H(S)$
is defined as $\A^1/(\A^1-0)\simeq\Sigma\Gm\simeq (\P^1,\infty)$.

We denote by $L_\Zar$ and $L_\Nis$ the Zariski and Nisnevich  sheafifications on $\Pre(\Sm_S)$ respectively. The $\A^1$-localization $\Lhtp\colon \Pre(\Sm_S)\to\Pre(\Sm_S)$ (resp. motivic localization $L_\mot\colon \Pre(\Sm_S)\to\Pre(\Sm_S)$) is the localization functor onto the full subcategory of $\A^1$-invariant presheaves of spaces (resp.\ of $\A^1$-invariant Nisnevich sheaves). A morphism $f$ in $\Pre(\Sm_S)$ is called an $\A^1$-equivalence (resp.\ a motivic equivalence) if $\Lhtp(f)$ (resp.\ $L_\mot(f)$) is an equivalence.
 
 \section{Recollection on twisted K-theory}\label{sec:recall}

In this section we recall the main definitions related to twisted algebraic K-theory.

Let $R$ be a commutative ring. Then an associative algebra $A$ over $R$ has a structure of left module over the algebra $A\otimes_R A^\op$ given by $(x\otimes y)\cdot z=xzy$. In particular we have an associative algebra homomorphism $A\otimes_R A^\op\to \End_R(A)$.
\begin{prop}
    Let $R$ be a commutative ring and $A$ an associative algebra over $R$. Then the following are equivalent:
    \begin{enumerate}
            \item $A$ is a finitely generated projective module of positive rank and the map $A\otimes_R A^\op\to \End_R(A)$ is an isomorphism;
            \item $A$ is a finitely generated projective module of positive rank, it is projective as an $A\otimes A^\op$-module and its center is equal to $R$;
            \item $A$ is a finitely generated projective $R$-module and for any map $R\to \Oo$ where $\Oo$ is a strict henselian ring, there exists a $d\ge1$ and an isomorphism of $\Oo$-algebras $A\otimes_R\Oo\simeq \M_d(\Oo)$;
            \item $A$ is a finitely generated projective $R$-module and for every map $R\to k$ where $k$ is a separably closed field, there exists an integer $d\ge 1$ and an isomorphism of $k$-algebras $A\otimes_R k\simeq \M_d(k)$;
            \item $A$ is a finitely generated projective $R$-module and for every map $R\to k$ where $k$ is a field, the algebra $A\otimes_R k$ is a central simple algebra.
    \end{enumerate}
    We say that an algebra $A$ is an Azumaya algebra if it satisfies any of the above conditions.
\end{prop}
\begin{proof}
    The equivalence of (1) and (2) is \cite[Theorem~3.4]{deMeyer-Ingraham}. The equivalence of (2) with (3), (4) and (5) is \cite[Proposition~IV.1.2]{Milne:1980} together with the fact that property (2) is étale-local, and the Artin-Wedderburn theorem.
\end{proof}

\begin{defn}
    Let $X$ be a scheme. An \emph{Azumaya algebra} over $X$ is a sheaf of quasicoherent $\Oo_X$-algebras $\Aa$ such that for every open affine $\Spec R\subseteq X$ the $R$-algebra $\Aa(R)$ is Azumaya. Equivalently, $\Aa$ is Zariski-locally an étale twisted form of $\M_d(\Oo_X)$. The integer $d$ is called the \emph{degree} of the Azumaya algebra (and it is constant on each connected component of $X$).
    
    We denote by $\Az(X)$ (resp. $\Az_d(X)$) the groupoid of Azumaya algebras over $X$ (resp. of degree $d$) with $\Oo_X$-algebra isomorphisms and by $\Az$ the corresponding stack.
\end{defn}

    Recall that the group scheme $\PGL_d$ is defined as the quotient of $\GL_d$ by its center $\Gm$. By Skolem-Noether theorem \cite[Proposition~IV.1.4]{Milne:1980}, the action by conjugation $\GL_d\rightarrow\Aut_{\Alg}(\M_d(\Oo))$ factors through $\PGL_d$ and induces an isomorphism
    $\PGL_d\simeq\Aut_{\Alg}(\M_d(\Oo))$
    of group schemes. This fact implies the following presentation of the groupoid of Azumaya algebras on $X$, aka \'etale twisted forms of $\M_d(\Oo_X)$.

\begin{prop}
    The natural map $$\Map(X, \B_\et\PGL_d)\rightarrow\Az_d(X)$$ is an
    equivalence of groupoids, where the left side denotes the mapping space in
    \'etale sheaves of spaces.
\end{prop}

In particular, we get that $$\H^1_\et(X,\PGL_d)\simeq\pi_0\Az_d(X)$$ is the set of
    isomorphism classes of Azumaya algebras and
    $$\PGL_d(X)=\H^0_\et(X,\PGL_d)\simeq\pi_1(\Az_d(X),\M_d(\Oo_X))$$ is the group of
    automorphisms of the trivial Azumaya algebra $\M_d(\Oo_X)$.

\begin{defn}
    Two Azumaya algebras $\Aa$ and $\Bb$ on $X$ are called \emph{Brauer equivalent} if
    there are vector bundles $\Ee$ and $\Ff$ on $X$ such that
    $\Aa\otimes_{\Oo}\End(\Ee)\simeq \Bb\otimes_{\Oo}\End(\Ff)$
    as $\Oo_X$-algebras. The \emph{Brauer group} $\Br(X)$ is the quotient of the set of isomorphism classes of Azumaya algebras on $X$ by Brauer equivalence. We write $[\Aa]$ for the class of $\Aa$ in $\Br(X)$.
\end{defn}

\begin{defn}
An \emph{index} of an Azumaya algebra $\Aa$ over $X$, denoted $\ind_X(\Aa)$, is the gcd of degrees of Azumaya algebras over $X$ that are Brauer equivalent to $\Aa$ (for a non-connected scheme, the index is defined componentwise).
\end{defn}

    The central extension 
    $$1\rightarrow\Gm\rightarrow\GL_d\rightarrow\PGL_d\rightarrow 1$$ induces
    a sequence
    $$\H^1_\et(X,\GL_d)\rightarrow\H^1_\et(X,\PGL_d)\xrightarrow{\delta}\H^2_\et(X,\Gm).$$
    This sequence is exact in the sense that the fiber of the map
    $\delta:\H^1_\et(X,\PGL_d)\rightarrow\H^2_\et(X,\Gm)$ over
    $0\in\H^2_\et(X,\Gm)$ is the image of $\H^1_\et(X,\GL_d)$. 
Moreover, by~\cite[Proposition~44]{Serre-Cohomologie-Galoisienne} there is an injective map
    \[\Br(X)\hookrightarrow\H^2_\et(X,\Gm).\]
    
    An Azumaya algebra of degree $d$ is $d$-torsion in the Brauer group. Hence if $X$ is quasicompact, $\Br(X)$ is a torsion subgroup of
    $\H^2_\et(X,\Gm)$, and thus have an inclusions:
    \[
    \Br(X) \hookrightarrow \Br'(X):=\H^2_\et(X,\Gm)_{\mrm{tors}} \hookrightarrow \H^2_\et(X,\Gm).
    \]

 Next, we define twisted versions of vector bundles, perfect complexes etc. By abuse of notation, for $f \colon Y \to X$ and $\Aa \in \Az(X)$ we denote $f^*(\Aa) \in \Az(Y)$ also by $\Aa$, when $f$ is clear from the context. 
 
\begin{defn}\label{def:twisted}
Let $\Aa$ be an Azumaya algebra over a scheme $X$. We define the category of \emph{$\Aa$-twisted sheaves on $X$} as 
\[\QCoh^\Aa(X)=\Mod_\Aa(\QCoh(X))\]
(by convention, $\Mod_\Aa$ stands for left modules). 

By Grothendieck's descent theorem, the map
\[f \colon Y \to X \, \mapsto \, \QCoh^\Aa(Y)  \]
is an  \'etale sheaf of categories on $\Sch_S$, which gives the stack of categories $\QCoh^\Aa$.
\end{defn} 

\begin{rem}\label{rem:lieblich} Let $\alpha \in \H^2_\et(X,\Gm)$ be a class, then we can associate to it a $\Gm$-gerbe $\sX \rightarrow X$. Lieblich has defined an $\alpha$-twisted sheaf as a quasicoherent sheaf on $\sX$ such that the inertial action (on the left) of $\Gm$ agrees with the $\sO_{\sX}$-module structure. The comparison between Definition~\ref{def:twisted} and his definition can be parsed together from the comparison results in Lieblich's thesis \cite[Section 2.1.3]{lieblich-thesis} and in C\u{a}ld\u{a}raru's \cite[Theorem 1.3.7]{caldararu-thesis}.
\end{rem}

Let $f \colon Y \to X$ be an \'etale cover of $X$ such that $f^* \Aa \simeq M_d(\Oo_Y)$. Then there is an equivalence 
\begin{equation}\label{eq:trivializing QCoh}
\QCoh^\Aa (Y) \simeq \QCoh(Y),
\end{equation}
given by sending an $M_d(\Oo_Y)$-module $M$ to $(\Oo_Y^{\oplus d})^\vee \otimes_{M_d(\Oo_Y)} M$.

\begin{defn}
The groupoid of \emph{$\Aa$-twisted vector bundles on $X$}, denoted $\Vect^\Aa(X)$, is given by the $\Aa$-twisted sheaves on $X$ that are  \'etale-locally free, and their isomorphisms. Precisely, $V \in \QCoh^\Aa(X)$ is an $\Aa$-twisted vector bundle if there is an \'etale cover $f \colon Y \to X$ such that $f^* \Aa \simeq M_d(\Oo_Y)$ and $f^* V \in \QCoh(Y)$ is a vector bundle on $Y$, under the identification~\eqref{eq:trivializing QCoh}.

Similarly, the category of \emph{$\Aa$-twisted perfect complexes on $X$}, denoted $\Perf^\Aa(X)$, is given by complexes of $\Aa$-twisted sheaves on $X$ that are \'etale-locally perfect. These data form the corresponding stacks $\Vect^\Aa$ and $\Perf^\Aa$.
\end{defn} 

\begin{rem}
$\QCoh^\Aa$ can be interpreted as an \'etale twisted form of $\QCoh$. Indeed, since there is an isomorphism of $\QCoh$-modules:  \[\Aut_{\QCoh}(\QCoh) \simeq \Pic,\] the \'etale twisted forms of $\QCoh$ are in bijection with $H^1(S, \Pic) \simeq H^2(S, \Gm)$, and $\QCoh^\Aa$ corresponds to $[\Aa] \in H^2(S, \Gm)$. Similarly, $\Vect^\Aa$ and $\Perf^\Aa$ are \'etale twisted forms of $\Vect$ and $\Perf$ respectively. 
\end{rem}

\begin{rem}
If $\Aa$ and $\Bb$ are Brauer equivalent Azumaya algebras, they are Morita equivalent, i.e. $\QCoh^\Aa \simeq \QCoh^\Bb$. Hence the stacks $\QCoh^\Aa$, $\Vect^\Aa$, $\Perf^\Aa$ only depend on the Brauer equivalence class of $\Aa$, as well as the $\Aa$-twisted K-theory and the related constructions. We twist with respect to a representative in the Brauer equivalence class only for the sake of convenience. 
\end{rem}

The following lemma will be important for studying $\Aa$-twisted K-theory.

\begin{lem}\label{lem:affine twisted vb}
    Let $V$ be an $\Aa$-twisted sheaf. Then the following are equivalent:
    \begin{enumerate}
        \item $V$ is an $\Aa$-twisted vector bundle;
        \item For every affine subscheme $i\colon U\hookrightarrow X$ the $\Aa(U)$-module $V(U)$ is finitely generated and projective;
        \item $V$ is locally free of finite rank as an $\Oo_X$-module;
    \end{enumerate}
\end{lem}

\begin{proof}
The equivalence between (2) and (3) is \cite[Proposition~II.2.3]{deMeyer-Ingraham}. The equivalence of (1) with (3) follows because the property of being locally free of finite rank is fpqc-local by \cite[\href{https://stacks.math.columbia.edu/tag/05B2}{05B2}]{stacks}.
\end{proof}
 
 With these definitions at hand, we can define $\Aa$-twisted K-theory. 
 
 \begin{defn}
 Let $\Aa$ be an Azumaya algebra over a qcqs scheme $S$. The \emph{$\Aa$-twisted K-theory} is the presheaf of spaces on $\Sm_S$ defined by 
\[K^\Aa(U) = K(\Perf^\Aa(U)),\] where $\Perf^\Aa(U)$ is considered as a Waldhausen category. 

 Alternatively, $K^\Aa$ can be characterized as the Zariski sheaf of spaces that on affine S-schemes is given by $K^\Aa(U) = \Vect^\Aa(U)^\gp$, where $\gp$ stands for group completion (with respect to direct sum on $\Vect^\Aa(U)$). Such a characterization is possible because $\Vect^\Aa$ is a Zariski sheaf by Lemma~\ref{lem:affine twisted vb}.
 
 We define the \emph{$\Aa$-twisted K-theory groups} of a smooth $S$-scheme $U$ as
 \[K^\Aa_i(U) = \pi_i K^\Aa(U). \]
  \end{defn}
  
  \begin{rem} \label{rem: K-module}
  $K^\Aa$ is naturally a module over the presheaf of $\Einfty$-ring spaces $K$, since $\Perf^\Aa$ is a $\Perf$-module via tensor product.
  \end{rem}
  
 \begin{defn} We define the \emph{$\Aa$-twisted rank map} as the map of presheaves of spaces 
\[\rk^\Aa \colon \Vect^\Aa \to \Z\qquad P\mapsto \sqrt{\rk_\Oo\End_\Aa(P)}\]  
that sends a locally free $\Aa$-module $P$ to the square root of  the rank of $\End_\Aa(P)$ as $\Oo$-module. This is well-defined  because $\QCoh^\Aa$ is a $\QCoh$-module hence enriched in $\QCoh$, and when $\Aa\simeq M_n(\Oo)$ it coincides with the usual rank under the identification \eqref{eq:trivializing QCoh} . In particular the map $\rk^\Aa$ takes values in the sheaf $\Z$ of integers, because \'etale-locally it gives an integer. Note that $\rk^\Aa(\Aa) = d$ for $\Aa$ an Azumaya algebra of degree $d$.

The stack $ \Vect^\Aa_m$ of \emph{$\Aa$-twisted vector bundles of rank $m$} is defined as the fiber of $\rk^\Aa-m \colon \Vect^\Aa \to \Z$.

Since $\Z$ is a sheaf of groups, the monoid map $\rk^\Aa$ factors through $K^\Aa$, and we get 
\[\rk^\Aa \colon K^\Aa \to \Z, \]  
which in turn factors through $K_0^\Aa$. 
 \end{defn}
 
\begin{defn}\label{def:twz}The sheaf of \emph{$\Aa$-twisted integers} is the following subsheaf of $\Z$:
  \[\Z^\Aa = L_\Zar \Im (\rk^\Aa \colon K^\Aa \to \Z). \]
 \end{defn} 
  \begin{lem}\label{lem:L_zar K_0}
  The induced map of presheaves \[\rk^\Aa \colon K_0^\Aa \to \Z^\Aa \] is a Zariski-local equivalence. Moreover, for $R$ a local ring, the inclusion $ K_0^\Aa(R) \simeq \Z^\Aa(R) \subset \Z$ is given by \[\ind_R(\Aa) \cdot \Z \subset \Z .\]
  \end{lem}
  
  \begin{proof}
By Morita invariance of $K_0^\Aa$ and~\cite[Corollary~1]{demeyer}, the map $\rk^\Aa \colon K_0^\Aa \to \Z$ is an injective homomorphism on local rings, hence it identifies $L_\Zar K_0^\Aa$ with $\Z^\Aa$. Over a local ring $R$, the image is given by $\ind_R(\Aa) \cdot \Z \subset \Z$ again by~\cite[Corollary~1]{demeyer}.
  \end{proof}

 \begin{cor}
 The sheaf inclusion $\Z^\Aa \subset \Z$ becomes an isomorphism after tensoring with $\Q$. 
 \end{cor}

By Lemma~\ref{lem:L_zar K_0}, our sheaf $\Z^\Aa$ is isomorphic to the sheaf (with transfers) $\Z_\Aa$ of~\cite[Section~5.3]{levine-kahn}, and our subsheaf inclusion $\Z^\Aa \subset \Z$ is the reduced norm map $\mathrm{Nrd}$ of~\cite[Lemma~5.2.1]{levine-kahn}. 
The following properties of $\Z^\Aa$ are proved in~\cite{levine-kahn}.

\begin{lem}\label{lem: Z^A mot space}
Let $S$ be regular of finite type. Then $\Z^\Aa$ is an $\A^1$-invariant Nisnevich sheaf on $\Sm_S$. Moreover, $\Z^\Aa$ is a birational sheaf, i.e. locally-constant in the Zariski  topology.
\end{lem}

\begin{proof}
This is~\cite[Lemma~5.1.1]{levine-kahn}.
\end{proof}
 
 \begin{lem}\label{lem: Z^A transfers}
Let $S$ be regular of finite type. Then the sheaf  $\Z^\Aa$ has a canonical structure of Voevodsky's transfers, and $\Z^\Aa \subset \Z$ is an inclusion of sheaves with transfers on $\Sm_S$. 
\end{lem}

\begin{proof}
This is~\cite[Section~5.3]{levine-kahn}.
\end{proof}

\section{The motivic space of twisted K-theory}\label{sec:motspc}

Let  $S$ be a quasi-compact quasi-separated base scheme, $\Aa$ an  Azumaya algebra of degree $d$ over $S$.
We want to understand the motivic homotopy type of the twisted K-theory space $K^\Aa$, which is a sheaf on $\Sm_S$, given on any affine $S$-scheme $\Spec R$ by $\Vect^\Aa(R)^\gp$. First, we show that it is indeed a motivic space when $S$ is regular. 

\begin{lem}\label{lem:localization-sequence}
    Let $X$ be a quasi-compact quasi-separated scheme and $Z\subseteq X$ a closed subset such that $U=X - Z$ is quasi-compact. Then there is a functorial fiber sequence of spaces 
    \[K(\Perf^\Aa(X\textrm{ on }Z))\to K^\Aa(X)\to K^\Aa(U)\,.\]
\end{lem}
\begin{proof}
    The claim follows by tensoring the $\Perf(X)$-linear exact sequence
    \[\Perf(X\textrm{ on }Z)\to \Perf(X)\to \Perf(U)\]
    by $\Perf^\Aa(X)$ and using \cite[Theorem~4.8.5.16.(4)]{HA}.
\end{proof}

\begin{prop}\label{thm:Nis-descent and A1-invariance}
    The presheaf of spaces $K^\Aa$ satisfies Nisnevich descent. If moreover $S$ is regular and noetherian of
    finite Krull dimension, then $K^\Aa$ is
    $\A^1$-invariant.
\end{prop}

\begin{proof}
The claim about Nisnevich descent follows from Lemma~\ref{lem:localization-sequence} as in~\cite[Corollary~10.10]{TT} by the \'etale descent for $\Aa$-twisted perfect complexes.

Let us prove the $\A^1$-invariance. Let $d$ be the degree of $\Aa$ and $\SB(\Aa)$ be the Severi-Brauer variety associated to $\Aa$. By~\cite[Theorem~4.1]{Quillen:1973}, for every smooth $S$-scheme $X$ there is an isomorphism 
\[K(X\times_S \SB(\Aa))\simeq \prod_{i=0}^{d-1} K^{\Aa^{\otimes i}}(X)\,.\]
Since $\SB(\Aa)$ is a smooth $S$-scheme, the pullback map
\[p^*\colon K(X\times_S \SB(\Aa))\to  K(\A^1_X\times_S \SB(\Aa))\]
is an equivalence. Since the decomposition is natural, it follows that the pullback map induces an equivalence on each component, in particular on the first component, which is what we wanted to prove.
\end{proof}

Let $\Spec R$ be a smooth affine $S$-scheme. By Lemma~\ref{lem:affine twisted vb}, $\Vect^\Aa(R)$ is the groupoid of finitely generated projective $\Aa_R$-modules, i.e. summands of free $\Aa_R$-modules (by the usual abuse of notation, we write further on  $\Aa$ for $\Aa_R$). Hence, we have \[\Vect^\Aa(R)^\gp \simeq  \Vect^\Aa(R)[-\Aa].\]

By the McDuff-Segal group completion theorem~\cite{mcduff-segal} (specifically, we use \cite[Proposition 6]{Nikolaus}), we get an equivalence of $\Einfty$-spaces:
\[
\Vect^\Aa(R)^\gp \simeq \sVect^\Aa(R)^+,
\]
where $\sVect^\Aa(R) = \colim (\Vect^\Aa(R) \xrightarrow{+\Aa} \Vect^\Aa(R) \xrightarrow{+\Aa} \dots)$ is the groupoid of stable $\Aa$-twisted vector bundles on $\Spec R$, and $+$ stands for Quillen's +-construction.

Recall that +-construction, applied to $\sVect^\Aa(R)$, is invisible up to $\A^1$-homotopy by Robalo's criterion~\cite[Insight 1.1]{Robalo} (which originates from~\cite[Theorem 4.3]{Voevodsky:1998}).  More precisely, we have the following result.

\begin{prop} On affine schemes, there is  a canonical $\A^1$-equivalence of $\Einfty$-spaces:
\[
\sVect^\Aa \simeq \Vect^{\Aa,\gp}.
\]
\end{prop}

\begin{proof}
By~\cite[Proposition 5.1]{EHKSY4}, it's enough to show that on any affine $S$-scheme $\Spec R$ the cyclic permutation of $\Aa^3$ is $\A^1$-homotopic to identity.  The action of the cyclic group $C_3$ on $\Aa^3$ induces a map of stacks $BC_3 \to \Vect^\Aa$. This map factors through the stack $B \SL_3$, and the group scheme $\SL_3$ is $\A^1$-connected, so the claim follows. 
\end{proof}

\begin{cor}\label{cor:K^A vs sVect^A}
There is a canonical $\A^1$-equivalence on affines:
\[
K^\Aa \simeq \sVect^\Aa.
\]
\end{cor}

\begin{rem}
Since twisted K-theory is Morita-invariant, it follows from Corollary~\ref{cor:K^A vs sVect^A} that for Brauer equivalent Azumaya algebras $\Aa$ and $\Bb$ there is a canonical $\A^1$-equivalence on affines between $\sVect^\Aa$ and $\sVect^\Bb$.
\end{rem}

\begin{prop}
For any affine scheme $\Spec R$, there is a \emph{non-canonical} equivalence of spaces  
\[\sVect^\Aa(R) \simeq K_0^\Aa(R) \times BGL(\Aa). \]
\end{prop}

\begin{proof}
This follows since $\pi_0 \sVect^\Aa(R) \simeq K_0^\Aa(R)$, and the automorphism group of any stable $\Aa$-twisted vector bundle is isomorphic to $GL(\Aa)$. The latter is explained for non-twisted vector bundles in~\cite[Section 2.1]{EHKSY3}, and the argument in \emph{loc.cit.} applies verbatim for twisted vector bundles, by replacing $R$ with $\Aa_R$. 
\end{proof}

The map of presheaves $\rk^\Aa \colon \Vect^\Aa \to \Z^\Aa$ induces a map 
\[\overline{\rk}^\Aa \colon \sVect^{\Aa} \to \Z^\Aa \]
that sends $V \in \Vect^\Aa$ in the $n$-th component to $ \rk^\Aa(V) - nd$ (this makes sense because the degree of an Azumaya algebra is stable under base change).

Consider the telescope \[ \Vect^\Aa_\infty = \colim (\Vect^\Aa_0 \xrightarrow{+\Aa} \Vect^\Aa_d \xrightarrow{+\Aa} \dots).\]

We get a fiber sequence of presheaves of spaces:
\begin{equation}\label{eq:non-split fiber seq}
\Vect^\Aa_\infty \to  \sVect^{\Aa} \xrightarrow{\overline{\rk}^\Aa} \Z^\Aa.
\end{equation}

\begin{rem}\label{rem:splits}
In the non-twisted case (i.e. when $\Aa$ is trivial), this fiber sequence splits canonically, since there is a map of sheaves $\Z \to \Vect$ sending an integer to a trivial vector bundle of the corresponding rank. However, this construction does not work in the twisted case. Indeed, while the $\Aa$-twisted rank map $\rk^\Aa\colon K_0^\Aa(R)\to \Z^\Aa(R)$ is surjective on all affine schemes $\Spec R$, the authors do not know if it is possible to choose the splitting in a natural way, and therefore whether the map of sheaves $K_0^\Aa\to \Z^\Aa$ has a section.

A related phenomenon can be observed in~\cite{ben-ben}, where the authors show that the index is not necessarily the image of an $\Aa$-twisted vector bundle (although it will be the image of an $\Aa$-twisted perfect complex).
\end{rem}

\begin{lem}
The fiber sequence~\eqref{eq:non-split fiber seq} splits Zariski-locally. 
\end{lem}

\begin{proof}
Let $R$ be a local ring. By~\cite[Corollary~1]{demeyer}, there is an Azumaya algebra $\Bb$ of degree $\ind_R(\Aa)$ in the Brauer equivalence class of $\Aa$ on $\Spec R$. Under the Morita equivalence $\Vect^\Aa \simeq \Vect^\Bb$, the $\Bb$-module $\Bb$ corresponds to some $\Aa$-twisted vector bundle $V$ such that $\rk^\Aa(V) = \ind_R(\Aa)$. By Lemma~\ref{lem:L_zar K_0}, the subgroup $\Z^\Aa(R) \subset \Z$ is given by the inclusion $\ind_R(\Aa) \cdot \Z \subset \Z$. Hence we can define the splitting of the surjection $\sVect^{\Aa}(R) \to \Z^\Aa(R)$ by sending $m \cdot \ind_R(\Aa)$ to $V^m$. 
\end{proof}

\begin{prop}\label{prop:fiber seq mot spaces}
After applying motivic localization, the sequence~\eqref{eq:non-split fiber seq} becomes a fiber sequence of motivic spaces in $\H(S)$. 
\end{prop}

\begin{proof}
It is enough to show the claim for the fiber sequence 
\begin{equation}\label{eq:non-split fiber seq Z}
\Vect^\Aa_\infty \to  \sVect^{\Aa} \xrightarrow{\overline{\rk}^\Aa} \Z.
\end{equation}
 
For a local ring $R$, the sequence 
\[\Lhtp \Vect^\Aa_\infty(R) \to \Lhtp \sVect^{\Aa}(R) \to \Lhtp \Z(R) \]
is a fiber sequence by~\cite[Proposition~5.4]{rezk-hate}, since the sheaf $\Lhtp \Z$ is constant in the simplicial direction. Hence, the sequence~\eqref{eq:non-split fiber seq Z} remains a fiber sequence after applying $L_\Zar \Lhtp$. Note that $L_\Zar \Lhtp \sVect^{\Aa} \simeq K^\Aa$ is a motivic space, and so is $L_\Zar\Lhtp \Z \simeq \Z$. Hence, the fiber $L_\Zar \Lhtp \Vect^{\Aa}_\infty$ is also a motivic space, and the claim follows. 
\end{proof}

Next, we want to replace $\Vect^\Aa_\infty$ with a more geometric object. 

\begin{defn}
The \emph{$\Aa$-twisted Grassmannian} is the presheaf $\Gr^\Aa_{m,n}$ on affine $S$-schemes sending $\Spec R$ to the set of $\Ff\subseteq\Aa^{\oplus n}_R$ of rank $m$ left $\Aa_R$-submodules such that $\Aa^{\oplus n}_R / \Ff$ is projective as a left $\Aa_R$-module. This construction generalizes Severi-Brauer schemes (in the case $m=1$), and has been studied in $K$-theoretic contexts by various authors \cite{levine-srinivas-weyman, panin, baek}.
\end{defn}

\begin{lem} \label{lem:sm-proper} The presheaf $\Gr^\Aa_{m,n}$ is represented by a smooth proper $S$-scheme.
\end{lem}

\begin{proof} The Azumaya algebra $\Aa$ has an underlying locally free sheaf $\Ee$. The classical Grassmannian $\Gr_{m,n}(\Ee)$ is represented by a smooth proper scheme. We observe that $\Gr^\Aa_{m,n}\subset \Gr_{m,n}(\Ee)$ is a closed subfunctor, hence it is represented by a proper subscheme. This scheme is smooth, because smoothness is an \'etale-local condition on the target of the forgetful map $\Gr^\Aa_{m,n} \to S$, and the classical Grassmannian $\Gr_{m,n}$ is smooth. 
\end{proof}

We define $\Gr^\Aa_{m,\infty} = \colim_n \Gr^\Aa_{m,n}$, with maps in the colimit induced by $(\id, 0) \colon \Aa^{\oplus n} \hookrightarrow \Aa^{\oplus (n+1)}$. We stabilize further and set 
\[\Gr^\Aa_\infty = \colim (\Gr^\Aa_{0,\infty} \xrightarrow{+\Aa} \Gr^\Aa_{d,\infty} \xrightarrow{+\Aa} \dots).\]

\begin{prop}\label{prop:twisted Gr vs Vect}
The forgetful maps $\Gr^\Aa_{m,\infty} \to \Vect^\Aa_m$ and $\Gr^\Aa_{\infty} \to \Vect^\Aa_\infty$ are $\A^1$-equivalences on affines. 
\end{prop}

\begin{proof}
When $\Aa$ is trivial, this is shown in~\cite[Proposition 4.7]{hjnty}. The argument applies verbatim for general $\Aa$, by replacing $A$ with $\Aa_A$ in \emph{loc.cit.}
\end{proof}

Altogether, we get the following result, which gives a geometric interpretation of the twisted K-theory as a motivic space.

\begin{thm}\label{thm:fiber seq K^A}
There is a fiber sequence of motivic spaces
\[\Gr^\Aa_\infty \to  K^{\Aa} \xrightarrow{\overline{\rk}^\Aa} \Z^\Aa.\]
\end{thm}

\begin{proof}
This is a combination of Proposition~\ref{prop:fiber seq mot spaces}, Corollary~\ref{cor:K^A vs sVect^A} and Proposition~\ref{prop:twisted Gr vs Vect}.
\end{proof}

\section{The motivic spectrum of twisted K-theory}\label{sec:motsptw}

In this section, we assume that the base scheme $S$ is regular and of finite Krull dimension. First, we show that the $\Aa$-twisted K-theory satisfies Bott periodicity.

\begin{prop}
    We have an isomorphism of $K(\Pp^1_S)$-modules:
    $K^\Aa(\Pp^1_S) \simeq K(\Pp^1_S)\otimes_{K(S)} K^\Aa(S)$.
\end{prop}

\begin{proof}
The tensor product map \[\Perf(\Pp^1_S) \otimes_{\Perf(S)} \Perf^\Aa(S) \to \Perf^\Aa(\Pp^1_S)\] is an isomorphism \'etale-locally, hence an isomorphism of stacks.
    Consider the Beilinson orthogonal decomposition of $\Perf(S)$-modules:
    \[\Perf(\Pp^1_S)\simeq\langle\Perf(S),\Perf(S)\rangle.\] Upon tensoring the decomposition with
    $\Perf^\Aa(S)$, we get
  \[\Perf^\Aa(\Pp^1_S)\simeq\langle\Perf^\Aa(S),\Perf^\Aa(S)\rangle.\]
  The proposition then follows by additivity of K-theory.
\end{proof}

\begin{cor}\label{cor:P1}
    We have $\Omega_{\Pp^1} K^\Aa\simeq K^\Aa$, i.e. $K^\Aa$ satisfies motivic Bott periodicity.
\end{cor}

Corollary~\ref{cor:P1} gives an infinite $\Pp^1$-delooping of $K^\Aa$, and so we can define the motivic spectrum of twisted K-theory.

\begin{defn}
The \emph{$\Pp^1$-spectrum of $\Aa$-twisted K-theory} is defined as follows (with bonding maps specified by Corollary~\ref{cor:P1}):
   \[\KGL^\Aa = (K^\Aa, K^\Aa, \dots) \in \SH(S) . \]
\end{defn}

\begin{rem}
Since $K^\Aa$ is a $K$-module by Remark~\ref{rem: K-module}, and the action is compatible with Bott periodicity, the motivic spectrum $\KGL^\Aa$ is a $\KGL$-module.

Similarly, $\kgl^\Aa$, the effective cover of $\KGL^\Aa$, is naturally a $\kgl$-module. In~\cite[Corollary~4.3]{tom-fflat-cancellation}, it was shown that a $\kgl$-module structure is equivalent to a structure of (coherent) pushforwards along finite locally free maps of schemes. Under this equivalence, the structure of transfers on $\kgl^\Aa$ comes from finite locally free covariance of $\Perf^\Aa$.
\end{rem}

Let us write $KH^\Aa(X)$ for the $\A^1$-localization of the Bass K-theory spectrum $K^B$ of $\Perf^\Aa(X)$. That is
\[KH^\Aa(X)=\colim_{n\in\Delta^\op} K^B(\Perf^\Aa(X\times\Delta^n))\,.\]
When $\Aa$ is trivial $KH^\Aa$ is the cohomology theory represented by $\KGL$. We want to show that the same is true for a general $\Aa$.
\begin{prop}\label{prop:KH-represents}
    Let $X$ be a quasi-compact quasi-separated scheme and $\Aa$ an Azumaya algebra over $X$. Then there is an equivalence of spectra
    \[\map_{\SH(X)}(\Sigma^\infty X_+, \KGL^\Aa)\simeq KH^\Aa(X)\]
\end{prop}
\begin{proof}
    The statement follows from \cite[Corollaire~2.11]{Cisinski} as in the proof of \cite[Théorème~2.20]{Cisinski}.
\end{proof}


One of the first achievements of motivic homotopy theory was a complete proof of cdh descent for homotopy $K$-theory \cite{Cisinski}. The key insight is that any motivic spectrum which is stable under base change satisfies cdh descent. We now prove a base change property for twisted K-theory spectrum to establish its cdh descent. To do so, we will first show that the sheaf $\Z^\Aa$ satisfies base change. We note that cdh descent for twisted homotopy $K$-theory can also be established using methods of Land-Tamme \cite{land-tamme} as explored in the twisted setting by Stapleton \cite{stapleton}. We will later give a \emph{motivic refinement} of the cdh descent of twisted homotopy $K$-theory in the sense that we will endow it with a motivic filtration whose associated graded pieces do enjoy cdh descent.

\begin{prop}\label{prop:pull_Z^A}
Let $f \colon T \to S$ be a morphism. Then the canonical map $f^* \Z^\Aa \to \Z^{f^* \Aa}$ is an isomorphism of Zariski sheaves on $\Sm_T$. 
\end{prop}

\begin{proof}
Since $f^*\Z^\Aa$ and $\Z^{f^*\Aa}$ are 0-truncated, it suffices to check that the map is an isomorphism on stalks.

The sheaf $f^*\Z^\Aa$ can be identified with the sheafification of the presheaf on $\Sm_T$ sending
\[[Z\to T]\mapsto \colim_{Z\to W\to S} \Z^\Aa(W)\]
where the colimit is indexed by all factorizations $Z\to W\to S$ of $Z\to T\to S$ where $W\to S$ is a smooth $S$-scheme. Let $Z\in\Sm_T$ and choose a point $z\in Z$, then the stalk of $f^*\Z^\Aa$ at $z$ is therefore
\[(f^*\Z^\Aa)_z\simeq \colim_{U\ni z} \colim_{U\to W\to S} \Z^\Aa(W)\simeq \colim_{\Spec \Oo_{Z,z}\to W\to S} \Z^{\Aa}(W)\]
where the last colimit is indexed through all factorizations of $\Spec \Oo_{Z,z}\to Z\to S$ through a smooth $S$-scheme $W$. Note that  $\Z^{\Aa}$ commutes with filtered colimits, as it is the Zariski sheafification of $K_0^\Aa$ by Lemma~\ref{lem:L_zar K_0}. Hence we can replace $W$ with $\Oo_{W,w}$ (where $w$ is the image of $z$) and rewrite this colimit as
\[(f^*\Z^\Aa)_z\simeq \colim_{\Spec \Oo_{Z,z}\to \Spec R\to S} \Z^{\Aa}(\Spec R)\simeq\colim_{\Spec \Oo_{Z,z}\to \Spec R\to S}K^{\Aa}_0(R)\]
where the colimit now ranges through all factorizations through an essentially smooth local $S$-scheme $\Spec R$ such that the map $R\to \Oo_{Z,z}$ is a local ring homomorphism. Moreover, the stalk of $\Z^{f^*\Aa}$ at $z$ is precisely $K_0^{\Aa}(\Oo_{Z,z})$, and the map we need to prove is an isomorphism is the pullback map
\[\colim_{\Spec \Oo_{Z,z}\to \Spec R\to S}K^{\Aa}_0(R)\to K^{\Aa}_0(\Oo_{Z,z})\,.\]
We will prove separately that it is injective and surjective.

For surjectivity it suffices to show that every $\Aa$-twisted vector bundle $V$ over $\Oo_{Z,z}$ can be obtained by pullback from some essentially smooth local ring $R$ over $S$. But, by choosing a set of generators as an $\Aa$-module, we can obtain it as a pullback from $\Gr^{\Aa}_{\Oo_{Z,z}}$ and therefore from $\Gr^\Aa_{\Oo_{S,s}}$ where $s$ is the image of $z$ in $S$. Since the $\Aa$-twisted Grassmannian is smooth by Lemma~\ref{lem:sm-proper}, this proves surjectivity.

Let us now prove injectivity. Suppose we have an essentially smooth local ring $R$ and a class $[V]-[V']\in K^\Aa_0(R)$ that is sent to 0 by the pullback along $f\colon \Spec \Oo_{Z,z}\to \Spec R$. Up to adding a suitable free $\Aa$-twisted vector bundle to both $V$ and $V'$, this means that there exists an isomorphism $\varphi\colon f^*V\simeq f^*V'$. Therefore the map $\Spec \Oo_{Z,z}\to \Spec R$ factorizes as
\[\Spec \Oo_{Z,z}\to \operatorname{Iso}_\Aa(V,V')\to \Spec R\]
where $\operatorname{Iso}_\Aa(V,V')$ is the smooth affine scheme of Lemma~\ref{lem:Iso-is-smooth-scheme} below. Therefore, by taking $\Spec R'$ to be the localization of $\operatorname{Iso}_\Aa(V,V')$ at the image of $z$, there is a factorization
\[\Spec \Oo_{Z,z}\to \Spec R' \to \Spec R\]
where $R'$ is an essentially smooth $S$-scheme where $V$ and $V'$ become isomorphic. Hence the class $[V]-[V']$ becomes 0 in the colimit.
\end{proof}

\begin{lem}\label{lem:Iso-is-smooth-scheme}
    Let $S$ be a qcqs scheme, $\Aa$ an Azumaya algebra over $S$ and $V$ and $W$ be two $\Aa$-twisted vector bundles. Consider the presheaf $\operatorname{Iso}_\Aa(V,W)$ sending every $S$-scheme $p\colon T\to S$ to the set $\{\varphi\colon p^*V\to p^*W\}$ of isomorphisms of $\Aa$-twisted vector bundles over $T$. Then $\operatorname{Iso}_\Aa(V,W)$ is represented by a smooth affine $S$-scheme.
\end{lem}
\begin{proof}
    Since the prestack of $\Aa$-modules is an fpqc stack by \cite[\href{https://stacks.math.columbia.edu/tag/023T}{023T}]{stacks}, the presheaf $\operatorname{Iso}_\Aa(V,W)$ is an fpqc sheaf. Moreover étale-locally we can find a trivialization of $\Aa$, so we can assume $\Aa$ to be trivial. But then, étale-locally, $\operatorname{Iso}_\Aa(V,W)$ is either empty (if $\rk V\neq \rk W$) or a $\GL_{\rk V}$-torsor. In either case it is represented by a smooth affine $S$-scheme.
\end{proof}

Using Proposition~\ref{prop:pull_Z^A}, we can deduce the base change property for $\KGL^\Aa$.

\begin{thm} \label{thm:pull} Let $S$ be a regular scheme of finite Krull dimension, and $\Aa$ an Azumaya algebra over $S$. The formation of the motivic twisted $K$-theory spectrum
\[
p\colon X \rightarrow S \, \mapsto \, KGL^{p^*\Aa}
\]
defines a Cartesian section of the Cartesian fibration $\int \SH \rightarrow \Sch_S$. In particular, the twisted homotopy $K$-theory spectrum is a cdh sheaf and satisfies Milnor excision.
\end{thm}

\begin{proof} The first statement follows once we know that for any morphism of $S$-schemes $f\colon X \rightarrow Y$, with structure mas $p_X\colon X \rightarrow S, p_Y\colon Y \rightarrow S$ the canonical comparison map in $\SH(X)$
\[
f^*KGL^{p_X^*\Aa} \rightarrow KGL^{p_Y^*\Aa}
\]
is an equivalence of motivic spectra. By construction of $\KGL^\Aa$, it's enought to check that the comparison map $f^*K^{p_X^*\Aa} \rightarrow K^{p_Y^*\Aa}$ is an equivalence of motivic spaces in $\H(X)$. By Theorem~\ref{thm:fiber seq K^A}, $K^\Aa$ as a motivic space is an extension of $\Z^\Aa$ by $\Gr_\infty^\Aa$. Then the claim follows from Proposition~\ref{prop:pull_Z^A}, because the twisted Grassmannians are stable under pullbacks. 

The base change property implies that $\KGL^\Aa$ is a cdh sheaf by~\cite[Proposition~3.7]{Cisinski}, and that it satisfies Milnor excision by~\cite[Corollary~2]{EHIK}.
\end{proof}

\section{Twisted motivic cohomology}\label{sec:motcohtw}

In this section, we compute the slices of $\Aa$-twisted K-theory, which allows us to construct the $\Aa$-twisted motivic spectral sequence.  Previously, Kahn and Levine have computed the slices in the case when the base scheme $S = \Spec k$ is the spectrum of a perfect field $k$~\cite[Theorem~6.5.5]{levine-kahn}. In this section, we will first provide a different proof of this computation (Theorem~\ref{thm:s_0KGL=HZ}) and then generalize it to the case when the Azumaya algebra is not pulled back from the base field (Proposition~\ref{thm:s_0KGL-general-base}).

To define $\Aa$-twisted motivic cohomology, we employ the theory of framed transfers, which was developed in~\cite{voevodsky2001notes, garkusha2014framed, EHKSY1} and other works. We summarize here the main structural results that we use.

\begin{enumerate}

\item Every cohomology theory represented by a motivic spectrum  has a unique structure of framed transfers, i.e. a structure of a presheaf on the $\infty$-category of framed correspondences $\Corr^\fr(\Sm_S)$ \cite[Theorem~18]{framed-loc}. More precisely,  
\[\SH(S) \simeq \SH^\fr(S),\] 
where $\SH^\fr(S)$ is the $\infty$-category of framed motivic spectra.

\item The framed suspension spectrum $\Sigma^\infty_\fr \colon \H^\fr(S) \to \SH^\fr(S)$, going from framed motivic spaces to (framed) motivic spectra, is fully faithful on grouplike objects, when $S$ is a perfect field~\cite[Theorem~3.5.14]{EHKSY1}. \\

\item There is a canonical  functor $\Corr^\fr(\Sm_S) \to \Cor_S$ from the  $\infty$-category of framed correspondences to the category of Voevodsky's finite correspondences~\cite[Section~5.3]{EHKSY1}. In particular, every presheaf with Voevodsky's transfers has canonical framed transfers.

\end{enumerate}

Recall that the presheaf of abelian groups $\Z^\Aa$ is a motivic space with a canonical structure of Voevodsky's finite transfers by Lemmas~\ref{lem: Z^A mot space} and~\ref{lem: Z^A transfers}, hence it represents a framed motivic space, which we also denote  $\Z^\Aa$.

\begin{defn}\label{def:twisted-mot}
We define the \emph{$\Aa$-twisted motivic cohomology spectrum} to be 
\[H\Z^\Aa := \Sigma^\infty_\fr\Z^\Aa \in \SH^\fr(S) \simeq \SH(S).\]

When $\Aa$ is trivial, the motivic spectrum $H\Z^\Aa$ coincides with the Spitzweck motivic cohomology spectrum, constructed in~\cite{SpitzweckHZ}, by~\cite[Theorem~21]{framed-loc}. Moreover, $H\Z^\Aa$ is stable under base change by Proposition~\ref{prop:pull_Z^A}, and in particular, the corresponding cohomology theory satisfies cdh descent by \cite[Proposition~3.7]{Cisinski}. Therefore it makes sense to say that $H\Z^\Aa$ represents $\Aa$-twisted Spitzweck motivic cohomology.
\end{defn} 

We recall the definition of the slice filtration in motivic homotopy theory, constructed in~\cite{Voevodsky:2002c}. 
\begin{defn}Let $S$ be a scheme. Then the $\infty$-category $\SH(S)^{\eff}$ of \emph{effective motivic spectra} (resp. the $\infty$-category $\SH_{S^1}(S)^{\eff}$ of \emph{effective $S^1$-spectra}) is the localizing subcategory of $\SH(S)$ (resp. $\SH_{S^1}(S)$) generated by the objects of the form $\Sigma^\infty_\T X_+$ (resp. $\Sigma^\infty_{S^1}X_+$) for $X\in \Sm_S$. For every $n\in\Z$ (resp. $n\in \N$) we call $\Sigma^n_\T\SH(S)^{\eff}$ the $\infty$-category of \emph{$n$-effective motivic spectra} and we write $f_n$ for the left adjoint to the inclusion $\Sigma^n_\T\SH(S)^{\eff}\subseteq\SH(S)$. Since $\Sigma^{n+1}_\T\SH(S)^{\eff}\subseteq\Sigma^n_\T\SH(S)^{\eff}$, for every $E\in \SH(S)$ we obtain a functorial $\Z$-indexed tower
\[ \cdots \to f_2 E\to f_1E\to \to \cdots \to E\]
called the \emph{slice tower} of $E$. Each layer $s_iE\coloneqq\cofib(f_{i+1}E\to f_iE)$ is called the \emph{$i$-th slice} of $E$. 
\end{defn}

We will use the following, more explicit description of 0-th slices. Consider the functor $L^0_{\bir}\colon\H^\fr(k)\to \H^\fr(k)$, which is the localization generated by the birational open embeddings of smooth varieties \cite[Section~3]{tom-elden-slices}. The following result follows from the main insight of~\cite{tom-elden-slices}.

\begin{prop}\label{lem:birational==0-slice} Let $k$ be a perfect field and $X \in \H^\fr(k)$ be a framed motivic space. 
\begin{enumerate}
\item The map
$X \rightarrow L^0_{\bir}X$
is an equivalence after applying $s_0\Sigma^{\infty}_\fr$.
\item Assume that $L^0_{\bir}X$ is grouplike, then $\Sigma^{\infty}_\fr L^0_{\bir}X$ is a 0-slice, i.e., we have an equivalence
\[
\Sigma^{\infty}_\fr L^0_{\bir}X \simeq s_0\Sigma^{\infty}_\fr X.
\]
In particular, any $L^0_{\bir}$-local grouplike object in $\H^{\fr}(k)$ becomes a 0-slice upon applying the functor $\Sigma^{\infty}_\fr$.
\end{enumerate}
\end{prop}

\begin{proof} We will write $F:\H(k)_\ast\to \H^{\fr}(k)$ for the left adjoint to the functor $\H^{\fr}(k)\to \H(k)_\ast$ that forgets the framed transfers. To prove (1), consider the commutative diagram
\[
\begin{tikzcd}
\H(k)_{\ast} \ar{r}{F} \ar[swap]{d}{\Sigma_{S^1}^\infty} & \H^{\fr}(k) \ar{d}{\Sigma^{\infty}_{\fr}}\\
\SH^{S^1}(k) \ar[swap]{r}{\Sigma^{\infty}_{\Gm}} & \SH(k)^{\eff}
\end{tikzcd}\,.
\]
According to \cite[Lemma 13.(5)]{tom-elden-slices}, the $L^0_{\bir}$-equivalences in $\H(k)_\ast$ are sent to $s_0$-equivalences in $\SH(k)$ (since their cofibers are sent to $1$-effective spectra). Moreover the $L^0_{\bir}$-equivalences in $\H^{\fr}(k)$ are, by definition, generated under colimits by the image under $F$ of the $L^0_{\bir}$-equivalences of $\H(k)_\ast$. Therefore the commutativity of the above diagram and the fact that all functors in it preserve all colimits, imply that the $L^0_{\bir}$-equivalences in $\H^{\fr}(k)$ are sent to $s_0$-equivalences in $\SH(k)$. But this is the claim of (1).

We now prove (2). After (1), we need only to prove that $\Sigma^{\infty}_\fr L^0_{\bir}X$ is a 0-slice, assuming that $L^0_{\bir}X$ is grouplike. Since $\Sigma^{\infty}_\fr L^0_{\bir}X$ is effective, it suffices to show that its first effective cover is trivial, i.e., that the homotopy sheaves $\underline{\pi}_{\ast}(\Sigma^{\infty}_{\fr}L^0_{\bir}X)_{-1}$ are zero.  For $\ast < 0$, this follows since the image of $\Sigma^{\infty}_{\fr}$ lands inside connective objects for the homotopy t-structure in $\SH(k)$. For $\ast\geq0$, we use that $\Sigma^{\infty}_{\fr}$ is fully faithful on grouplike objects by the cancellation theorem~\cite{agp}. Hence  $\underline{\pi}_\ast(\Sigma^\infty_{\fr} L^0_{\bir}X)_{-1}$ is the sheaf associated to the presheaf
\[U\mapsto \pi_\ast\Map_{\SH(k)}(\Sigma^\infty_{\fr}F(\G_{m}\wedge U_+),\Sigma^\infty_{\fr}L^0_{\bir}X)\simeq \pi_\ast\Map_{\H(k)^{\fr}}(F(\G_{m}\wedge U_+),L^0_{\bir}X)\,.\]
However, since $\G_m\times U\to \A^1\times U$ is a birational map, we get
\[\Map_{\H(k)^{\fr}}(F(\G_{m}\wedge U_+),L^0_{\bir}X)\simeq \Map_{\H(k)^{\fr}}(F(\A^1\wedge U_+),L^0_{\bir}X)\simeq \ast\,,\]
which is what we wanted to show.
Since an $L^0_{\bir}$-local object $X$ satisfies $X \simeq L^0_{\bir}X$, the last claim of (2) follows.
\end{proof}

\begin{rem}\label{rem:fr} It is plausible that $\Sigma^{\infty}_{\fr}L^0_{\bir}X \simeq s_0\Sigma^{\infty}_{\fr}X$ for $X \in \H^\fr(k)$ such that  $L^0_{\bir}X$ is \emph{not necessarily grouplike}, which would then give a geometric formula for the 0-th slice of $\Sigma^{\infty}_{\fr}X$. Since we do not need this result, we leave it to the interested reader.
\end{rem}

Our goal in this section is to prove that the $n$-th slice $s_n\KGL^\Aa$ is equivalent to $\Sigma^n_\T H\Z^\Aa$. To do so, we will use the techniques for slice computations developed in~\cite{tom-elden-slices}. Recall that we write $\kgl^\Aa$ for the effective cover of $\KGL^\Aa$.
\begin{lem}\label{lem:kgl-is-very-effective}
    Let $X$ be a scheme essentially smooth over a field and $\Aa$ an Azumaya algebra over $X$. Then the motivic spectrum $\kgl^\Aa$ is very effective.
\end{lem}
\begin{proof}
    By Theorem~\ref{thm:pull}, \cite[Lemma~B.1 and Proposition~B.3]{norms} it suffices to show the thesis when $X$ is the spectrum of a field. The motivic spectrum $\kgl^\Aa$ is effective by definition, so by \cite[Theorem~2.3]{HoyoisMGL} it suffices to check that for $Y$ an essentially smooth henselian local scheme, the spectrum
    \[\map_{\SH(X)}(\Sigma^\infty Y_+,\kgl^\Aa)\simeq\map_{\SH(X)}(\Sigma^\infty Y_+,\KGL^\Aa)\simeq K^B(\Perf^\Aa(Y))\]
    is connective. But by \cite[Theorem~4.1]{bernardara} this spectrum is a summand of the Bass K-theory of  $\SB(\Aa)$, the Severi-Brauer scheme of $\Aa$. Since $\SB(\Aa)$ is smooth over $Y$, it is a regular Noetherian scheme and so its Bass K-theory is connective.
\end{proof}

The following computation is the key geometric input for identifying the slices of $\KGL^\Aa$.

\begin{lem}\label{lem:Gr-L^0_bir-contractible}
    The motivic space $\Gr_\infty^\Aa$ is $L^0_{\bir}$-contractible.
\end{lem}
\begin{proof}
    Let $d$ be the degree of $\Aa$. It suffices to show that $\Gr^\Aa_{md,n}$ is $L^0_{\bir}$-contractible for every $m$ and $n\ge md$, since $\Gr_\infty^\Aa=\colim_{n\ge md}\Gr_{md,n}^\Aa$.
    
    We want to show that $\Gr^\Aa_{md,n}$ contains an open dense subscheme that is Zariski-locally $\A^1$-contractible.
        Let $\Hom_\Aa(\Aa^m,\Aa^{n-m})$ be the sheaf of $\Aa$-linear maps $\Aa^m\to \Aa^{n-m}$. Then there is a map \[\Hom_\Aa(\Aa^m,\Aa^{m-n})\to \Gr^\Aa_{md,n}\] sending a map $f:\Aa^m\to \Aa^{m-n}$ to its graph in $\Aa^m\times \Aa^{n-m}\simeq \Aa^n$. This map is a dense open embedding, since étale-locally $\Aa$ is trivial, and in that case this map is one of the standard charts of the Grassmannian $\Gr_{md,nd}$. We observe that 
        \[\Hom_\Aa(\Aa^m,\Aa^{n-m}) \simeq  \Aa^{m(n-m)}.\] 
        Therefore $\Gr^\Aa_{md,n}$ is birational to a vector bundle, and thus it is $L^0_{\bir}$-trivial. 
\end{proof}

Using the above lemma we can present an alternative proof of \cite[Theorem~6.5.5]{levine-kahn}.
\begin{thm}\label{thm:s_0KGL=HZ}
    Let $k$ be a perfect field, $\Aa$ be an Azumaya algebra over $k$ and $f\colon S\to \Spec k$ be an essentially smooth $k$-scheme. Then the map in $\SH(S)$
    \[\Sigma^\infty_\fr \Vect^{f^*\Aa}\to \Sigma^\infty_\fr\Z^{f^*\Aa} = H\Z^{f^*\Aa}\,,\]
    induced by the $\Aa$-twisted rank map $\rk^\Aa \colon \Vect^\Aa \to \Z^\Aa$, can be identified with the canonical map
    \[\kgl^{f^*\Aa}\to s_0(\kgl^{f^*\Aa})\,.\]
    Therefore for every $n\in\Z$ there is an equivalence
    \[s_n\KGL^{f^*\Aa}\simeq s_n\KGL\otimes_{H\Z}H\Z^{f^*\Aa}\simeq\Sigma^n_\T H\Z^{f^*\Aa}\,.\]
\end{thm}
\begin{proof}
    By \cite[Remark~4.20]{HoyoisMGL}, \cite[Theorem~B.4]{norms},  Proposition~\ref{prop:pull_Z^A} and Theorem~\ref{thm:pull}, all terms in the statement of the theorem are stable under essentially smooth base change. Therefore we can assume that $S=\Spec k$ is the spectrum of a perfect field.
    
    By the motivic recognition principle~\cite[Theorem~3.5.14]{EHKSY1} and Lemma~\ref{lem:kgl-is-very-effective} we have
    \[\kgl^\Aa \simeq \Sigma^\infty_\fr \Vect^\Aa \simeq \Sigma^\infty_\fr K^\Aa.\]
    
    By Lemma~\ref{lem: Z^A mot space}, $\Z^\Aa$ is a birational sheaf. Hence the motivic spectrum $H\Z^\Aa$ belongs to the image of the localizing functor $s_0$ on effective motivic spectra by Lemma~\ref{lem:birational==0-slice}.(2).
    
    Let now $\overline{\kgl}^\Aa$ be the fiber of the map $\kgl^\Aa\to H\Z^\Aa$, induced by $\rk^\Aa$. It remains to be proven that $s_0(\overline{\kgl}^\Aa)=0$. Since for any smooth $S$-scheme $U$, the map of sheaves of spectra
    \[K^\Aa(U)=\map(\Sigma^\infty U_+,\kgl^\Aa)\to \map(\Sigma^\infty U_+,H\Z^\Aa)=H(\Z^\Aa(U))\]
    is Zariski-locally 0-connected, it follows that $\kgl^\Aa$ and $\overline{\kgl}^\Aa$ are very effective. By Theorem~\ref{thm:fiber seq K^A} and \cite[Theorem~3.5.14]{EHKSY1} we see that
    \begin{equation}\label{eq:kgl-gr}
    \overline{\kgl}^\Aa\simeq \Sigma^\infty_\fr \Gr_\infty^\Aa,
    \end{equation}
    for some structure of framed transfers on $\Gr_\infty^\Aa$. 
    
   But now we have equivalences:
   \begin{equation*}
   s_0(\overline{\kgl}^\Aa) \simeq  s_0(\Sigma^\infty_\fr \Gr_\infty^\Aa) \simeq s_0(\Sigma^{\infty}_{\fr}L_0^{\bir}\Gr_\infty^\Aa) = 0.
   \end{equation*}
    Indeed, the first equivalence comes from~\eqref{eq:kgl-gr}, the second equivalence is Lemma~\ref{lem:birational==0-slice}.(1) and the third equivalence is Lemma~\ref{lem:Gr-L^0_bir-contractible} together with \cite[Lemma 13.(4)]{tom-elden-slices}.

    The final statement follows from Bott periodicity for $\KGL^\Aa$, see  Corollary~\ref{cor:P1}.
\end{proof}

\begin{cor}\label{cor:s_0KGL-imperfect}
    Let $S=\Spec k$ be the spectrum of a field of exponential characteristic $e$ and $\Aa$ be an Azumaya algebra over $k$. Then the map in $\SH(S)$
    \[\Sigma^\infty_\fr \Vect^\Aa [1/e]\to \Sigma^\infty_\fr\Z^\Aa [1/e]= H\Z^\Aa[1/e],\]
    induced by the $\Aa$-twisted rank map $\rk^\Aa \colon \Vect^\Aa \to \Z^\Aa$, can be identified with the canonical map
    \[\kgl^\Aa[1/e]\to s_0(\kgl^\Aa[1/e])\,.\]
   Moreover for every $n\in\Z$ there is an equivalence
    \[s_n\KGL^\Aa[1/e]\simeq s_n\KGL\otimes_{H\Z}H\Z^{\Aa}[1/e]\simeq\Sigma^n_\T H\Z^{\Aa}[1/e]\,.\]
\end{cor}
\begin{proof}
    Let $k^p$ be the perfection of $k$. Then by \cite[Corollary 2.1.7]{elden-adeel-perfection} the statement follows from Theorem~\ref{thm:s_0KGL=HZ} in $\SH(k^p)[1/e]$.
\end{proof}

Our next goal is to generalize Corollary~\ref{cor:s_0KGL-imperfect} to the case when the Azumaya algebra $\Aa$ is not pulled back from the base field. To begin we will show that $H\Z^\Aa$ is a 0-slice in this more general case.

\begin{defn}
We say that a Noetherian scheme $X$ is \emph{Grothendieck} if for every point $x\in X$ the local ring $\Oo_{X,x}$ is a G-ring \cite[\href{https://stacks.math.columbia.edu/tag/07GH}{Tag 07GH}]{stacks}. Every scheme of finite type over a Grothendieck scheme, in particular every scheme of finite type over a field, is again Grothendieck \cite[\href{https://stacks.math.columbia.edu/tag/07PX}{Tag 07PX}]{stacks}.
\end{defn}

\begin{lem}\label{lem:abelian} Let $R$ be a Grothendieck ring, and let $F\colon\Sm_R^{\op} \rightarrow \Sp$ be a Nisnevich sheaf of spectra. Let $\widetilde{F}\colon \mathrm{EssSm}_R^{\op} \rightarrow \Sp$ be its finitary extension to essentially smooth $R$-schemes. Assume that for all $x \in X$ where $X \in \Sm_R$, we have that $\widetilde{F}(\widehat{\Oo_{X,x}}) = 0$, where $\widehat{\Oo_{X,x}}$ is the completion of $\Oo_{X,x}$ at its maximal ideal. Then $F=0$. 
\end{lem}

\begin{proof} Since $F$ is, in particular, a Zariski sheaf, it then suffices to prove that for any smooth affine morphism $R \rightarrow S$, we have that $F(S) = 0$. Now, by \cite[\href{https://stacks.math.columbia.edu/tag/07PX}{Tag 07PX}]{stacks}, $S$ is a $G$-ring since it is a finite type extension of $R$. Therefore, for any prime ideal $\mathfrak{p}$ of $S$, the map $S_{\mathfrak{p}} \rightarrow \widehat{S_{\mathfrak{p}}}$ is a regular homomorphism. Popescu's theorem \cite[\href{https://stacks.math.columbia.edu/tag/07GC}{Tag 07GC}]{stacks} then proves that this map is ind-smooth and thus $\widetilde{F}(\widehat{\Oo_{X,x}})$ is well-defined. 

Now, since $F$ is a Nisnevich sheaf, it suffices to prove that $\widetilde{F}(S^h_{\mathfrak{p}}) = 0$. The map $S_{\mathfrak{p}} \rightarrow \widehat{S_{\mathfrak{p}}}$ factors through $S^h_{\mathfrak{p}} \rightarrow \widehat{S_{\mathfrak{p}}}$ by \cite[\href{https://stacks.math.columbia.edu/tag/06LJ}{Tag 06LJ}]{stacks}. By assumption, it then suffices to prove that $\widetilde{F}(S^h_{\mathfrak{p}}) \rightarrow \widetilde{F}(\widehat{S_{\mathfrak{p}}})$ is injective on homotopy groups. But now, since $S_{\mathfrak{p}}$ is by definition a noetherian local $G$-ring, $S^h_{\mathfrak{p}}$ is a $G$-ring by \cite[\href{https://stacks.math.columbia.edu/tag/07QR}{Tag 07QR}]{stacks}. Thus Popescu's theorem applies again to prove that $S^h_{\mathfrak{p}} \rightarrow \widehat{S_{\mathfrak{p}}}$ can be written as a filtered colimit of smooth ring maps $\{f_{\alpha}:S^h_{\mathfrak{p}} \rightarrow B_{\alpha}\}_{\alpha}$. By construction, for each $B_{\alpha}$ we have a morphism $B_{\alpha} \rightarrow \widehat{S_{\mathfrak{p}}} \rightarrow \kappa(\mathfrak{p})$ under $S^h_{\mathfrak{p}}$, i.e., a $\kappa(\mathfrak{p})$-section of the map $S^h_{\mathfrak{p}} \rightarrow B_{\alpha}$. Therefore, by \cite[Th\'eor\`eme~I.8]{Gruson}, we obtain a section of $S^h_{\mathfrak{p}} \rightarrow B_{\alpha}$ extending the section over $\kappa(\mathfrak{p})$. This means that the map $\pi_n\widetilde{F}(S^h_{\mathfrak{p}}) \rightarrow \pi_n\widetilde{F}(B_{\alpha})$ is injective for every $n\in\Z$. We conclude by observing that a filtered colimit of injective morphisms of abelian groups is injective.


\end{proof}

\begin{lem}\label{lemma:HZ-is-a-slice}
    Let $k$ be a field of exponential characteristic $e$, $X$ a regular Grothendieck $k$-scheme, and $\Aa$ an Azumaya algebra over $X$. Then the map 
    \[\Z^\Aa[1/e]\to \Omega^\infty H\Z^\Aa[1/e]\]
    is an equivalence of Nisnevich sheaves. 
    Furthermore, $H\Z^\Aa[1/e]$ is a 0-slice.
\end{lem}
\begin{proof}
    
    By Lemma~\ref{lem:abelian} applied to the fiber and the hypercompleteness of the Nisnevich topos of $X$, it suffices to prove the result after evaluating on $\widehat{\Oo_{X,x}}$ where $x\in X$ is a point of $X$. 
    
    Let now $i \colon \Spec\kappa(x) \hookrightarrow \Spec\widehat{\Oo_{X,x}}$ be the inclusion of the residue field. By the Cohen structure theorem \cite[\href{https://stacks.math.columbia.edu/tag/0C0S}{Tag 0C0S}]{stacks} there is an isomorphism $\widehat{\Oo_{X,x}}\simeq \kappa(x)[\![t_1,\dots,t_n]\!]$ with $n=\dim_x X$. In particular, there is an ind-smooth map $r\colon  \widehat{\Oo_{X,x}} \to \Spec \kappa(x)$ such that $r\circ i$ is the identity. Moreover, since $i^*$ is an equivalence on the Brauer group by \cite[Corollary~IV.2.13]{Milne:1980}, the Azumaya algebras $r^*i^*\Aa$ and $\Aa$ are Brauer equivalent and so $\Z^{r^*i^*\Aa}\simeq\Z^{\Aa}$. Therefore, by \cite[Lemma~A.7]{HoyoisMGL}, the map we want to prove is an equivalence is the pullback along $r$ of the map
    \[\Z^{i^*\Aa}[1/e]\to \Omega^\infty H\Z^{i^*\Aa}[1/e]\]
    in $\SH(\kappa(x))$. But this map is an equivalence by \cite[Remark~3.2.8]{elden-adeel-perfection}.

    Let us now prove that $H\Z^\Aa[1/e]$ is a 0-slice in $\SH(X)$. It suffices to show that for every smooth $X$-scheme $Y$ the map
    \[\Map_{\SH(X)}(\Sigma^\infty (Y\times\P^1)_+, H\Z^\Aa[1/e])\to \Map_{\SH(X)}(\Sigma^\infty Y_+,H\Z^\Aa[1/e])\]
    is an equivalence. As we already proved, this is equivalent to asking whether the map
    \[\Z^\Aa(Y\times \P^1)[1/e]\to \Z^\Aa(Y)[1/e]\]
    is an isomorphism, which is true by Lemma~\ref{lem: Z^A mot space}.
\end{proof}

Under the assumptions of Lemma~\ref{lemma:HZ-is-a-slice} we can compute the slices of $\KGL^\Aa$. This generalizes \cite[Theorem 1]{levine-kahn} as explained in the introduction.
\begin{thm}\label{thm:s_0KGL-general-base}
    Let $k$ be a field of exponential characteristic $e$, $X$ a regular Grothendieck $k$-scheme, and $\Aa$ an Azumaya algebra over $X$. Then for every $n\in\Z$ there is an equivalence
    \[s_n\KGL^\Aa[1/e]\simeq s_n\KGL\otimes_{H\Z}H\Z^{\Aa}[1/e]\simeq\Sigma^n_\T H\Z^{\Aa}[1/e]\,.\]
\end{thm}
\begin{proof}
Consider the canonical map of framed spaces $\Vect^\Aa \to K^\Aa$. By adjunction, it induces a map of motivic spectra $\Sigma^\infty_\fr\Vect^\Aa{}[1/e] \to \kgl^\Aa{}[1/e]$, which can be checked to be an equivalence Nisnevich-locally. Over the complete Noetherian local schemes $\Spec \widehat{\Oo_{X,x}}$ we argue as in the proof of Lemma~\ref{lemma:HZ-is-a-slice}. Indeed, our map is a pullback along an ind-smooth retraction $r\colon  \Spec \widehat{\Oo_{X,x}} \to \Spec \kappa(x)$ of the map
\[\Sigma^\infty_\fr\Vect^{i^* \Aa}[1/e] \to \kgl^{i^* \Aa}[1/e]\]
    in $\SH(\kappa(x))$, which is an equivalence by the motivic recognition principle~\cite[Theorem~3.5.14]{EHKSY1}, since $\kgl^{i^*\Aa}$ is very effective by Lemma~\ref{lem:kgl-is-very-effective}.
    
    We argue that the induced map 
    \[\kgl^\Aa[1/e] \simeq \Sigma^\infty_\fr\Vect^\Aa[1/e] \xrightarrow{\Sigma^\infty_\fr \rk^\Aa} H\Z^\Aa[1/e] \]
    is equivalent to taking the 0-th slice. By Lemma~\ref{lemma:HZ-is-a-slice} we know that $H\Z^\Aa[1/e]$ is a 0-slice. It remains to show that the fiber of $\Sigma^\infty_\fr \rk^\Aa$ is 1-effective after inverting $e$. Arguing as before we can reduce to the case of fields (by checking on the complete local rings of $X$, where it is pulled back from a field along an ind-smooth map). There it holds by Corollary~\ref{cor:s_0KGL-imperfect}.
\end{proof}

Let $k$ be a field and $X$ be a smooth $k$-scheme. Recall that the \emph{$n$-th Bloch cycle complex} is the simplicial abelian group $z^n(X)$ which in degree $d$ is given by 
\[\bigoplus_{Z\subseteq X\times \Delta^d} \Z\cdot{} [Z]\]
where the sum is over all the irreducible closed subsets of codimension $d$ of $X\times \Delta^d$ that intersect all faces of $X\times \Delta^d$ properly. The face maps in the simplicial abelian group are given by intersections with the faces of $X\times \Delta^d$.  By \cite[Theorem~6.4.2]{Levine:2008}, there is a homotopy equivalence of the mapping spectrum $\map_{\SH(k)} (\Sigma^\infty X_+,\Sigma^n_\T H\Z)$ with (the $H\Z$-module spectrum represented by) $z^n(X)$.

\begin{defn}
Let $\Aa$ be an Azumaya algebra over $X$. The \emph{$\Aa$-twisted Bloch cycle complex} is the subcomplex $z^n_\Aa(X)\subseteq z^n(X)$ of those cycles $\sum_Z n_Z \cdot  [Z]$ such that $n_Z\in \Z^\Aa(Z)\subseteq\Z$.
\end{defn}

\begin{prop}\label{prop:cycle-model-of-Z^A}
    Let $k$ be a perfect field and $\Aa$ an Azumaya algebra over $k$. Then for every smooth $k$-scheme $X$ the map of mapping spectra
    \[\map_{\SH(k)}(\Sigma^\infty_\T X_+,\Sigma^n_\T H\Z^\Aa)\to \map_{\SH(k)}(\Sigma^\infty X_+,\Sigma^n_\T H\Z)\]
    identifies the left hand side with $z^n_\Aa(X)\subseteq z^n(X)$.
\end{prop}
\begin{proof}
    By Theorem~\ref{thm:s_0KGL=HZ} it follows that $s_n(\Sigma^nH\Z^\Aa)\simeq \Sigma^n H\Z^\Aa$. Therefore, from \cite[Corollary~5.3.2 and Theorem~9.0.3]{Levine:2008} applied with $p=n$ and $E=\Sigma^n_\T H\Z^\Aa$, it follows that the map we want to study is induced by a map of simplicial spectra
    \[\bigoplus_{Z\subseteq X\times\Delta^\bullet} (s_0H\Z^\Aa)(Z)\cdot [Z]\to \bigoplus_{Z\subseteq X\times\Delta^\bullet} (s_0H\Z)(Z)\cdot [Z]\]
    where the sums are indexed by the closed subspaces of $X\times\Delta^\bullet$ of codimension $n$ intersection all faces properly. But since $s_0H\Z^\Aa\simeq H\Z^\Aa$ and the right hand side is canonically identified with Bloch cycle complex, this proves the thesis.
\end{proof}

\begin{rem}\label{rem:nis-local}
    The authors know how to prove Proposition~\ref{prop:cycle-model-of-Z^A} only when $\Aa$ is pulled back from a perfect field. However the argument in the proof of Lemma~\ref{lemma:HZ-is-a-slice} shows that, up to inverting the exponential characteristic, this is true Nisnevich locally whenever $X$ is smooth over a field. Therefore the identification of $\Omega^\infty\Sigma^n_\T H\Z^\Aa$ with $z^n_\Aa(X)$ holds Nisnevich locally in this generality. It seems therefore reasonable to conjecture that this equivalence is true globally.
\end{rem}

\begin{defn}
The $(i, n)$-th \emph{$\Aa$-twisted motivic cohomology group} $H^i_{\mot}(X;\Z^\Aa(n))$ is defined to be $\pi_{2n-i}\map_{\SH(X)}(\Sigma^\infty_\T X_+,\Sigma^n_\T H\Z^\Aa)$.
When $\Aa$ is an Azumaya algebra over a perfect field $k$, the group $H^i_{\mot}(X;\Z^\Aa(n))$ is isomorphic to the $(i-2n)$-th cohomology group of the $\Aa$-twisted Bloch cycle complex $z^n_\Aa(X)$ by Proposition~\ref{prop:cycle-model-of-Z^A}, and therefore to the twisted Chow groups of \cite[Definition~5.6.3]{levine-kahn}.
\end{defn}

By considering the slice spectral sequence for $\KGL^\Aa$ we obtain the following corollary, which recovers the spectral sequence of \cite[Corollary 6.1.4]{levine-kahn} when $\Aa$ is pulled back from the base field.

\begin{cor}\label{cor:motivic-ss}\label{cor:main}
    Let $k$ be a field of exponential characteristic $e$, $X$ a regular Grothendieck $k$-scheme, and $\Aa$ an Azumaya algebra over $X$. Then there is a strongly convergent spectral sequence
    \[E^{p,q}_2=H^{p-q}_{\mot}(X;\Z^\Aa[1/e](-q))\Rightarrow K_{-p-q}^\Aa(X)[1/e]\,.\]
    Moreover if the field $k$ is perfect and $\Aa$ is pulled back from $k$, then one can avoid inverting $e$, so that the spectral sequence has the  shape
    \[E^{p,q}_2=H^{p-q}_{\mot}(X;\Z^\Aa(-q))\Rightarrow K_{-p-q}^\Aa(X)\,.\]
\end{cor}
\begin{proof}
    The slice spectral sequence for $\KGL^\Aa$ converges to $K^\Aa(X)$ by Proposition~\ref{prop:KH-represents} and Proposition~\ref{thm:Nis-descent and A1-invariance}. Moreover its $E_2$-page can be identified with motivic cohomology by Theorem~\ref{thm:s_0KGL=HZ} when $\Aa$ is pulled back from a perfect field and by Proposition~\ref{thm:s_0KGL-general-base} after inverting $e$.
    
    Finally, the strong convergence of the spectral sequence follows as in \cite[Remark~6.1.2]{levine-kahn} from the fact that $\kgl^\Aa$ is very effective, which was proved in Lemma~\ref{lem:kgl-is-very-effective}.
\end{proof}

\bibliographystyle{alphamod}
\let\mathbb=\mathbf
{\small
\bibliography{twisted}
}

\end{document}